\documentclass[a4paper,12pt]{article}
\usepackage{amsfonts,amsmath,amssymb,amsthm,verbatim,placeins,color,pgfplots,comment,placeins,caption,subcaption,enumerate,bm}
\usepackage{graphicx}
\usepackage{geometry}
\usepackage[pdfauthor={DG},pdfstartview={FitH},colorlinks=true,citecolor=blue]{hyperref}
\usepackage[english]{babel}
\usepackage[sort&compress]{natbib}

\newtheorem{theorem}{Theorem}[section]
\newtheorem{proposition}{Proposition}[section]
\newtheorem{lemma}{Lemma}[section]

\theoremstyle{remark}
\newtheorem{remark}{Remark}[section]
\theoremstyle{definition}
\newtheorem{definition}{Definition}[section]
\newtheorem{example}{Example}[section]

\theoremstyle{theorem}
\newtheorem*{assumption}{Assumption}

\theoremstyle{remark}
\newtheoremstyle{myremark}{}{}{\color{blue}\small}{}{\color{blue}\bfseries}{}{ }{}

\theoremstyle{myremark}

\newcommand{\E}{\mathbb{E}} %for expectation
\renewcommand{\Re}{\operatorname{Re}} %for real part
\newcommand{\R}{{\mathbb R}}
\newcommand{\N}{{\mathbb N}}

\newcommand{\C}{{\mathbb C}}

\newcommand{\T}{{\mathcal{T}}}
\newcommand{\mS}{{\mathcal{S}}}
\newcommand{\mH}{{\mathcal{H}}}

\makeatletter
\renewcommand*{\@fnsymbol}[1]{\ensuremath{\ifcase#1\or *\or \mathparagraph\or \ddagger\or
        \mathsection\or \mathparagraph\or \|\or **\or \dagger\dagger
        \or \ddagger\ddagger \else\@ctrerr\fi}}
\makeatother

\begin{document}
%\setcounter{page}{1000}
%%%%%%%% TITLE PAGE %%%%%%%%%%%%%%%%%%%%
%\thispagestyle{empty}
%\vspace*{10cm}
%\noindent This is a preprint of the paper:\\
%
%\noindent Grahovac, D.~(2020) Multifractal processes: Definition, properties and new examples, \emph{Chaos, Solitons \& Fractals}, \textbf{134}(109735).
%
%
%\noindent URL: \url{https://www.sciencedirect.com/science/article/pii/S0960077920301375}\\
%DOI: \url{https://doi.org/10.1016/j.chaos.2020.109735}
%\newpage
%\setcounter{page}{1}
%%%%%%%%%%%%%%%%%%%%%%%%%%%%%%%%%%%%%%%%

\renewcommand*{\thefootnote}{\fnsymbol{footnote}}

\begin{center}
\Large{\textbf{Multifractal processes: definition, properties and new examples}}\\
\bigskip
%\large{\today}\\
\bigskip
Danijel Grahovac$^1$\footnote{dgrahova@mathos.hr}\\
\end{center}

\bigskip
\begin{flushleft}
\footnotesize{
$^1$ Department of Mathematics, University of Osijek, Trg Ljudevita Gaja 6, 31000 Osijek, Croatia}
\end{flushleft}

\bigskip

\textbf{Abstract: }
We investigate stochastic processes possessing scale invariance properties which we refer to as multifractal processes. The examples of such processes known so far do not go much beyond the original cascade construction of Mandelbrot. We provide a new definition of the multifractal process by generalizing the definition of the self-similar process. We establish general properties of these processes and show how existing examples fit into our setting. Finally, we define a new class of examples inspired by the idea of Lamperti transformation. Namely, for any pair of infinitely divisible distribution and a stationary process one can construct a multifractal process.

\bigskip

\section{Introduction}
Multifractality may refer to a variety of properties which are usually used to describe objects possessing some type of scale invariance. The term itself has its true meaning in the analysis of local regularity properties of measures and functions. When it comes to stochastic processes, multifractality usually refers to models that exhibit nonlinear scaling of moments in time. More precisely, for some range of values of $q$ and $t$, the $q$-th absolute moment of a process $\{X(t)\}$ at $t$ can be written in the form
\begin{equation}\label{momscal:motivating}
  \E|X(t)|^q = c(q) t^{\tau (q)},
\end{equation}
where the so-called scaling function $q \mapsto \tau(q)$ is nonlinear. This contrasts the case of self-similar processes for which $\tau$ is a linear function. The property can also be based on the moments of increments of the process and can also be assumed to hold only asymptotically for $t\to 0$, in contrast to the exact scaling in \eqref{momscal:motivating}. For more details see \cite{mandelbrot1997mmar}, \cite{riedi2003multifractal} and the references therein. The processes satisfying \eqref{momscal:motivating} or some variant of it have received considerable attention in variety of applications like turbulence, finance, climatology, medical imaging, texture classification (see e.g.~\cite{mandelbrot1997mmar}, \cite{bacry2008continuous}, \cite{robert2008hydrodynamic}, \cite{duchon2012forecasting}, \cite{lovejoy2013weather},  \cite{abry2015irregularities}, \cite{pavlov2016multifractality}, \cite{kalamaras2017multifractal}, \cite{laib2018multifractal} and the references therein). 

The scaling of moments as in \eqref{momscal:motivating} is usually a consequence of a more general scaling property and so it is not the best candidate for a defining property. In this paper we specify multifractality as a property of the finite dimensional distributions of the process. Recall that the process $\{X(t)\}$ is self-similar if for every $a>0$ there exists $b>0$ such that
\begin{equation*}
\{X(at)\} \overset{d}{=} \{bX(t)\},
\end{equation*}
where $\{\cdot\} \overset{d}{=} \{\cdot\}$ denotes the equality of the finite dimensional distributions of two processes. The basic idea of our approach is to generalize the definition of a self-similar process. This idea can be traced back to Mandelbrot (see e.g.~\cite{mandelbrot1997mmar}), who used the following property as a motivation for the scaling of moments: for a process $\{X(t), \, t \in \T\}$ there is a set $\Lambda \subset (0,\infty)$ such that for every $\lambda \in \Lambda$ there exists a positive random variable $M_{\lambda}$ independent from the process $\{X(t)\}$ such that
\begin{equation}\label{mfdef:motivating}
  X(\lambda t) \overset{d}{=} M_\lambda X(t), \quad \forall t \in \T.
\end{equation}
However, to our knowledge, such generalizations of self-similarity have never been studied systematically. One can also find \eqref{mfdef:motivating} stated in the sense of equality of finite dimensional distributions, that is
\begin{equation}\label{mfdef:motivatingfdd}
  \{X(\lambda t) \}_{t\in \T} \overset{d}{=} \{ M_\lambda X(t)\}_{t\in \T}.
\end{equation}
Such properties are also referred to as exact scale invariance, exact stochastic scale invariance or statistical self-similarity (see e.g.~\cite{allez2013lognormal}, \cite{bacry2003log}, \cite{barral2014exact}). A typical example of a process satisfying \eqref{mfdef:motivatingfdd} are cascade processes constructed in \cite{bacry2003log}, \cite{barral2002multifractal} and \cite{muzy2002multifractal} . A class of examples has also been given in \cite{veneziano1999basic} called independent stationary increment ratio processes, but these actually satisfy a variant of \eqref{mfdef:motivating} and not the corresponding variant of \eqref{mfdef:motivatingfdd} (see Subsection \ref{subsec:examples} for details). 

Imposing that relation \eqref{mfdef:motivating} holds in the sense of equality of finite dimensional distributions seems unnecessarily restrictive as it requires for $\lambda \in \Lambda$ the existence of a \emph{single} random variable $M_\lambda$ such that $(X(\lambda t_1), \dots, X(\lambda t_n)) =^d (M_\lambda X(t_1), \dots,\allowbreak M_\lambda X(t_n))$ for any choice of $t_1,\dots,t_n\in \T$ and $n\in \N$. Instead we will require only that the random factors corresponding to each time point are identically distributed (see Definition \ref{def:mf} for details). This simple generalization of \eqref{mfdef:motivatingfdd} will allow us to obtain a whole class of new examples and to build a theory that generalizes self-similarity in a natural way. 

In Section \ref{sec2} we start with the definition and its ramifications. We show that, in contrast to self-similar processes, the scaling property of multifractal processes cannot hold for every scale $\lambda > 0$ as it then reduces to self-similarity. Also, the equality of finite dimensional distributions as in \eqref{mfdef:motivatingfdd} cannot hold over $(0,\infty)$ except when the process is self-similar. From these facts we identify restrictions that should be imposed on the set of scales and the time sets involved.

In Section \ref{sec3} we investigate general properties of multifractal processes. We show an intimate connection between scaling factors and infinitely divisible distributions. More precisely, by appropriately transforming a family of scaling factors indexed by $\lambda \in \Lambda$, one obtains a process which has one-dimensional distributions as some L\'evy process. From this relation scaling of moments as in \eqref{momscal:motivating} is a straightforward consequence. We also show that the cascade processes provide an example where the process obtained from the family of scaling factors is not a L\'evy process although its one-dimensional marginals correspond to some L\'evy process. 

In Section \ref{sec4}, using the idea of Lamperti transformation, we define a class of new multifractal processes which we call \emph{$L$-multifractals}. To any pair of L\'evy process and stationary process there corresponds one $L$-multifractal process. 

Properties of $L$-multifractal processes are investigated in Section \ref{sec5}. It is shown that these process do not have stationary increments in general. However, by restricting the time set and by carefully choosing the corresponding L\'evy process and stationary process, a process with second-order stationary increments may be obtained.

Although most papers on multifractal processes deal with variations of \eqref{momscal:motivating}, an exception is the work of \cite{veneziano1999basic} which aims at providing a general treatment of processes satisfying a variant of \eqref{mfdef:motivatingfdd}. A close inspection of the proofs there shows that all the examples given actually satisfy a variant of \eqref{mfdef:motivating} and not \eqref{mfdef:motivatingfdd}. Moreover, the scaling factors are stated to be in one-to-one correspondence with L\'evy processes, which would exclude the case of cascade processes as we show in this paper. Let us also mention that the construction of the family of scaling factors provided there via L\'evy process necessarily involves bounded time interval and cannot be extended to $(0,\infty)$. Another line of development that generalizes the cascade construction is the multiplicative chaos theory that can be seen as a continuous analog of Mandelbrot's $\star$-scale invariance (see \cite{rhodes2014gaussian} for details). However, the random measures obtained in this way do not possess exact scale invariance in general, except when they reduce to the cascade case. See \cite{allez2013lognormal}, \cite{barral2013gaussian}, \cite{rhodes2014gaussian} and the references therein.

In \cite{veneziano1999basic}, the processes satisfying \eqref{mfdef:motivatingfdd} are referred to as stochastically self-similar (see also \cite{gupta1990multiscaling}). Although this is a more meaningful term than multifractal, we prefer the later as it is now widespread in the literature.

\section{The definition of a multifractal process}\label{sec2}
We first introduce some notation and assumptions. All the processes considered will have values in $\R$. In what follows, $=^d$ stands for the equality in law of two random variables, while $\{\cdot\} =^d \{\cdot\}$ denotes the equality of finite dimensional distributions of two stochastic processes. If the time set is not specified in this equality than it is assumed that it holds over the whole time set where the processes are defined. A process $\{X(t), \, t \in \T\}$ is said to be nontrivial if $X(t)$ is not a constant a.s.~for every $t\in \T$. Every process considered will be assumed to be not identically null so that there is $t\in \T$, $t\neq 0$, such that $P(X(t)\neq 0)>0$.  

As elaborated in the introduction, \eqref{mfdef:motivatingfdd} may be overly restrictive and it does not provide a natural generalization of self-similar processes. For this reason, we introduce the following definition of multifractality.

\begin{definition}\label{def:mf}
A stochastic process $X=\{X(t), \, t \in \T\}$ is \textbf{multifractal} if there exist sets $\Lambda \subset (0,\infty)$ and $\mS \subset \T$ such that for every $\lambda \in \Lambda$ there exists a family of identically distributed positive random variables $\{M(\lambda,t), \, t \in \mathcal{S} \}$, independent of $\{X(t)\}$ such that
\begin{equation}\label{mfdef:prop}
  \{ X(\lambda t) \}_{t \in \mS} \overset{d}{=} \{ M(\lambda, t) X(t) \}_{t \in \mS}.
\end{equation}
\end{definition}

To clarify further, \eqref{mfdef:prop} means that for every choice of $t_1,\dots,t_n \in \mS$, $n \in \N$ it holds that
\begin{equation*}
  \left(X(\lambda t_1),\dots,X(\lambda t_n) \right) \overset{d}{=} \left( M(\lambda, t_1) X(t_1), \dots, M(\lambda, t_n) X(t_n) \right).
\end{equation*}
Clearly, this generalizes the equality \eqref{mfdef:motivatingfdd} where $M(\lambda,t)=M_\lambda$ does not depend on $t$. This and the case of self-similar processes inspired the condition stated in the definition that for every fixed $\lambda \in \Lambda$
\begin{equation}\label{Mstatmarg}
  M(\lambda, t_1) \overset{d}{=} M(\lambda, t_2), \quad \forall t_1, t_2 \in \mS.
\end{equation}
We will refer to random field $M=\{M(\lambda,t), \, \lambda \in \Lambda, t \in \mS\}$ as the \textbf{family of scaling factors}.

\begin{remark}
Definition \ref{def:mf} involves several sets:
\begin{itemize}
	\item $\T$ is the time set where the process is defined,
	\item $\Lambda$ is a set of scales for which \eqref{mfdef:prop} holds
	\item $\mS$ is a subset of $\T$ over which the equality of finite dimensional distributions \eqref{mfdef:prop} holds.
\end{itemize}	
Such level of generality will prove to be important later on. We will assume that $\T$ is one of the following: $\T=(0,\infty)$, $\T=[0,\infty)$ or $\T=(0,T]$, $\T=[0,T]$ for some $T>0$. Note also that $0 \in \mS$ makes the unduly restriction that $X(0)=0$ a.s.~if $M(\lambda,t)$ is not a constant a.s.
\end{remark}

\begin{example}
Suppose $X$ is multifractal with $\Lambda=(0,\infty)$, $\mS=\T$ and for every $\lambda \in \Lambda$, $M(\lambda,t)$ is deterministic. Due to \eqref{Mstatmarg}, $t \mapsto M(\lambda,t)$ is constant, say $m(\lambda)$. The definition then reduces to the classical definition of self-similarity (see e.g.~\cite{embrechts2002selfsimilar}, \cite{pipiras2017long}). Furthermore, if $X$ is nontrivial and right continuous in law (meaning that for every $t_0>0$, $X(t) \to^d X(t_0)$ as $t \downarrow t_0$), then there exists a unique $H\in \R$ called the Hurst parameter such that $m(\lambda)=\lambda^H$ (see \cite[Section 8.5.]{bingham1989regular} and the references therein). We shortly say $X$ is $H$-ss. Typically $\T=(0,\infty)$ and if $\T=[0,\infty)$, then $H\geq 0$. Furthermore, $H>0$ implies $X(0)=0$ a.s.~and $H=0$ if and only if $X(t)=X(0)$ a.s.~for every $t>0$ (\cite{embrechts2002selfsimilar}).

In this setting, \eqref{mfdef:motivating} corresponds to a concept of marginally self-similar process as defined in \cite[Section 8.5.]{bingham1989regular}. Adopting this terminology, a process satisfying \eqref{mfdef:prop} only for the one-dimensional distributions may be referred to as marginally multifractal. Note that this is equivalent to \eqref{mfdef:motivating} when $\mS=\T$.
\end{example}

\subsection{The set $\Lambda$}

We start our analysis of Definition \ref{def:mf} by investigating how the scaling property \eqref{mfdef:prop} affects the size of the set $\Lambda$. It is known that the scaling of moments \eqref{momscal:motivating} cannot hold for every $t>0$ with $\tau$ being nonlinear (see e.g.~\cite{mandelbrot1997mmar} and \cite{muzy2013random}). In our setting this correspond to restrictions on the set $\Lambda$. 

We start by showing that if the scaling is deterministic, then typically $\Lambda=(0,\infty)$ and $\mS=\T$.

\begin{proposition}\label{prop:LamSwhendeterministic}
Suppose $X=\{X(t), \, t \in \T\}$ is a multifractal process such that $M(\lambda,t)=m(\lambda)$ is deterministic for every $\lambda \in \Lambda$. If $\Lambda$ contains an interval, then \eqref{mfdef:prop} holds for any $\lambda\in (0,\infty)$. If additionally $\mS$ contains an interval, then \eqref{mfdef:prop} holds with $\mS=\T$, hence $X$ is self-similar.
\end{proposition}

\begin{proof}
Let $\Lambda'$ denote the set of $\lambda$ for which \eqref{mfdef:prop} holds. If $\lambda \in \Lambda'$, then from $\{X(t)\}_{t \in\mS} \overset{d}{=} \{m(\lambda) X(t/\lambda)\}_{t \in\mS}$ we have $\{X(t/\lambda)\}_{t \in\mS} \overset{d}{=} \{1/m(\lambda) X(t)\}_{t \in\mS}$ implying that $1/\lambda \in \Lambda'$. Furthermore, if $\lambda_1, \lambda_2 \in \Lambda'$, then since 
\begin{equation}\label{prop1proofmultiplciative}
\{X(\lambda_1 \lambda_2 t)\}_{t \in\mS} \overset{d}{=} \{m(\lambda_1) X(\lambda_2  t)\}_{t \in\mS}\overset{d}{=} \{m(\lambda_1) m(\lambda_2) X(t)\}_{t \in\mS},
\end{equation}
we have $\lambda_1 \lambda_2 \in \Lambda'$. Hence, $\Lambda'$ is a multiplicative subgroup of $(0,\infty)$ containing $\Lambda$. Since it has positive measure it must be $\Lambda'=(0,\infty)$ (see e.g.~\cite[Corollary 1.1.4]{bingham1989regular}).

Let $t \in \mS$ be such that $P(X(t)\neq 0)>0$. Such $t$ exists, as if $X(t)=0$ a.s.~for every $t\in\mS$, then $X$ is identically null since $X(s)\overset{d}{=} m(s/t) X(t)$ for every $s\in\T$. By using \cite[Lemma 1.1.1]{embrechts2002selfsimilar} we conclude from \eqref{prop1proofmultiplciative} that $m(\lambda_1 \lambda_2)=m(\lambda_1)m(\lambda_2)$.

We now show that \eqref{mfdef:prop} can be extended to $\mS'$ of the form $\mS'= \alpha \mS=\{\alpha s : s \in \mS\}$ for any $\alpha>0$. Indeed, for every $\lambda \in (0,\infty)$ we have
\begin{align*}
\{ X(\lambda t) \}_{t \in \mS'} &= \{ X(\lambda \alpha t/\alpha) \}_{t \in \mS'} = \{ X(\lambda \alpha s) \}_{s \in \mS} \overset{d}{=} \{ m(\lambda \alpha) X(s) \}_{s \in \mS} \\
&= \{ m(\lambda) m(\alpha) X(s) \}_{s \in \mS} \overset{d}{=} \{ m(\lambda) X(\alpha s) \}_{s \in \mS} \overset{d}{=} \{ m(\lambda) X(t) \}_{t \in \mS'},
\end{align*}
which proves the claim.
\end{proof}

Proposition \ref{prop:LamSwhendeterministic} shows that typically, deterministic scaling may be extended to any scale $\lambda \in (0,\infty)$. The next proposition provides a sort of converse showing that the random scaling cannot hold for any scale since then it reduces to deterministic scaling. To show this, we introduce a further assumption on the process and make use of the Mellin transform (see Appendix \ref{appendix:mellin}).

For the argument bellow, we would need to assume that the domain of the Mellin transform applied to $X(t)$ does not degenerate into imaginary axis. For the process $X=\{X(t), \, t\in \T\}$, let
\begin{equation*}%\label{qoverline}
  \overline{q} \left(X\right) = \sup \left\{ q \geq 0 : \E|X(t)|^q < \infty \text{ for all } t \in \T \right\}.
\end{equation*}
We will assume that $\overline{q}\left(X\right)>0$, so that for every $t\in \T$ the Mellin transform of $X(t)$ is defined at least for $0 \leq \Re z < \overline{q} \left(X\right)$. It is worth noting that this is a very mild assumption.

\begin{proposition}\label{prop:nomfgen}
Suppose $X=\{X(t), \, t \in \T\}$ is a multifractal process with $\overline{q} \left(X\right)>0$. If $\Lambda=(0,\infty)$, then $M(\lambda,t)=m(\lambda)$ is a.s.~a constant for any $\lambda \in \Lambda'$, $\Lambda'=\{\lambda \in \Lambda : \text{ there exists } t \in \mS \text{ such that }\lambda t \in \mS\}$. In particular, if $\mS=\T$, then $X$ is self-similar.
\end{proposition}

\begin{proof}
Recall that we have assumed that for every process considered, there is $t\in \T$, $t\neq 0$, such that $P(X(t)\neq 0)>0$. Since $\{X(t)\}$ is multifractal, then this is true for every $\lambda t$, $\lambda \in \Lambda$, $t\in\mS$. There is at least one $t\in \mS$ such that $P(X(t)\neq 0)>0$ as otherwise $X$ would be identically null because of $X(s)\overset{d}{=} M(s/t, t) X(t)$. Since $\Lambda=(0,\infty)$ we conclude that $P(X(t)\neq 0)>0$ for every $t\in\T$, $t\neq 0$.

Let $\lambda \in \Lambda'$ and take $t \in \mS$ such that $\lambda t \in \mS$. Then $1/\lambda\in \Lambda$ and from $X(t)\overset{d}{=} M(1/\lambda,\lambda t) X(\lambda t)$ and $X(\lambda t)\overset{d}{=} M(\lambda,t) X(t)$ we have
\begin{align*}
  \mathcal{M}_{|X(t)|}(z) &= \mathcal{M}_{M(1/\lambda, \lambda t)}(z) \mathcal{M}_{|X(\lambda t)|}(z), \\
  \mathcal{M}_{|X(\lambda t)|}(z) &= \mathcal{M}_{M(\lambda, t)}(z) \mathcal{M}_{|X(t)|}(z).
\end{align*}
and
\begin{equation*}
  \mathcal{M}_{|X(t)|}(z) = \mathcal{M}_{M(1/\lambda,\lambda t)}(z) \mathcal{M}_{M(\lambda,t)}(z) \mathcal{M}_{|X(t)|}(z).
\end{equation*}
Here $\mathcal{M}_X$ denotes the Melilin transform of the random variable $X$ (see Appendix \ref{appendix:mellin}). Since $P(|X(t)| > 0)>0$, for real $z \in (0,\overline{q} \left(X\right))$ we have that $\mathcal{M}_{|X(t)|}(z) =\E|X(t)|^z >0$  and hence
\begin{equation}\label{e:proofp2:1}
  \mathcal{M}_{M(1/\lambda, \lambda t)}(z) \mathcal{M}_{M(\lambda,t)}(z) = 1.
\end{equation}
This uniquely determines the distribution (see Appendix \ref{appendix:mellin}) and if we have independent random variables $M_1\overset{d}{=}M(1/\lambda,\lambda t)$ and $M_2\overset{d}{=}M(\lambda,t)$, then \eqref{e:proofp2:1} implies that $M_1 M_2 = 1$ a.s. This is impossible unless $M_1$ and $M_2$ are constants a.s., so we conclude $M(\lambda,t)$ is a.s.~a constant. Due to \eqref{Mstatmarg}, $M(\lambda,t)=m(\lambda)$ for every $t\in \mS$. If $\mS=\T$, then, as we have assumed, $\mS$ is of the form $(0,\infty)$, $[0,\infty)$, $(0,T]$ or $[0,T]$. All of these imply $\Lambda'=(0,\infty)$, hence it follows that the process is self-similar.
\end{proof}

\begin{remark}
For $\Lambda'=\Lambda$ to hold in Proposition \ref{prop:nomfgen}, it is enough that $(0,u)\subset \mS$ or that $(u,\infty)\subset \mS$ for some $u>0$.
%The assumption on the set $\mS$ in Proposition \ref{prop:nomfgen} is clearly satisfied whenever $(0,1) \subset \mS$ or $(1, \infty) \subset \mS$. This covers all the examples appearing in this paper.
\end{remark}

\begin{remark}
Note that without further assumptions $Y_1 Z \overset{d}{=} Y_2 Z$ with $Z$ positive and independent of $Y_1$ and $Y_2$ does not necessarily imply that $Y_1\overset{d}{=} Y_2$. Indeed, in \cite[p.~506]{feller1971introduction}, one can find example of random variables such that $\widetilde{Y}_1+\widetilde{Z} \overset{d}{=} \widetilde{Y}_2 + \widetilde{Z}$ with $\widetilde{Z}$ independent of $\widetilde{Y}_1$ and $\widetilde{Y}_2$, but $\widetilde{Y}_1$ and $\widetilde{Y}_2$ do not have the same distribution. By taking exponentials one gets the counterexample for the product (see also \cite[Exercise 1.12.]{chaumont2012exercises}). This explains the assumption $\overline{q} \left(X\right)>0$ in Proposition \ref{prop:nomfgen}. We note that this condition can be replaced with the appropriate condition on moments of negative order. More precisely, we can assume $\underline{q}\left(X\right)<0$ where
\begin{equation}\label{qunderline}
  \underline{q} \left(X\right) = \inf \left\{ q \leq 0 : \E|X(t)|^q < \infty \text{ for all } t \in \T \right\}.
\end{equation}
\end{remark}

The crucial property for the proof of Proposition \ref{prop:nomfgen} is that every $\lambda \in \Lambda$ has its inverse element $1/\lambda$ in $\Lambda$. To enable random scaling, one has to consider $\Lambda$ being monoid and not group under multiplication. Thus, we will have in general two distinct classes of multifractal processes depending on whether $\Lambda=(0,1]$ or $\Lambda=[1,\infty)$. In \cite{veneziano1999basic}, the processes defined by the variation of property \eqref{mfdef:motivatingfdd} are referred to as contraction (resp.~dilation) stochastically self-similar in the case that corresponds to our $\Lambda=(0,1]$ (resp.~$\Lambda=[1,\infty)$).

\subsection{The set $\mS$}

To enable random scaling one has to make restrictions on the set $\mS$ too. Indeed, the next proposition shows that if $\Lambda=(0,1]$ and $[1,\infty) \subset \mS$, or if $\Lambda=[1,\infty)$ and $(0,1] \subset \mS$, then the scaling is necessarily deterministic and reduces to self-similarity. In particular, there is no random scaling if $\mS=(0,\infty)$. To prove this we will assume that the process $X$ under consideration defined on the probability space $(\Omega, \mathcal{F}, P)$ is jointly measurable, i.e.~$(t, \omega) \mapsto X(t,\omega)$ is $\mathcal{B}(\T) \times \mathcal{F}$-measurable (see also Remark \ref{rem:cauchyfuneqproof}).

\begin{proposition}\label{prop:nomfgenonS}
Suppose $\T=(0,\infty)$ or $\T=[0,\infty)$ and $X=\{X(t), \, t \in \T\}$ is a jointly measurable multifractal process such that $\overline{q} \left(X\right)>0$ and $\Lambda=(0,1]$ or $\Lambda=[1,\infty)$. If $\overline{\Lambda} \subset \mS$, $\overline{\Lambda}:=\{1/\lambda : \lambda\in \Lambda \}$, then $X$ is self-similar.
\end{proposition}

\begin{proof}
First note that $P(X(t)\neq 0)>0$ for every $t \in \mS$. Indeed, let $s \in \mS$ be such that $P(X(s)\neq 0)>0$. Such $s$ exists, as if $X(s)=0$ a.s.~for every $s\in\mS\supset \overline{\Lambda}$, then from $X(\lambda)\overset{d}{=} M(\lambda,1) X(1)$ we would have $X(s)=0$ a.s.~for every $s>0$ and $X$ would be identically null. For arbitrary $t\in \mS$, either $s/t \in \Lambda$ or $t/s \in \Lambda$. From $X(s)\overset{d}{=} M(s/t,t) X(t)$ or from $X(t)\overset{d}{=} M(t/s,s) X(s)$ we then conclude $P(X(t)\neq 0)>0$.

Let $q \in (0,\overline{q} \left(X\right))$. Given $t_1, t_2 \in (0,\infty)$, take $\lambda \in \Lambda$ such that $\lambda t_1, \lambda t_1 t_2 \in \Lambda$ and hence, by the assumptions, $1/(\lambda t_1) \in \mS$. Since $1/\lambda \in \mS$ we have from \eqref{mfdef:prop} that
\begin{equation}\label{prop:nomfgenonS:proof:eq1}
\begin{aligned}
X(t_1) &\overset{d}{=} M(\lambda t_1, 1/\lambda) X(1/\lambda),\\
X(1) &\overset{d}{=} M(\lambda, 1/\lambda) X(1/\lambda).
\end{aligned}
\end{equation}
By denoting
\begin{align*}
f_q(\lambda) &= \E |M(\lambda, t) |^q, \quad \lambda \in \Lambda,\\
g_q(t) &= \frac{\E |X(t) |^q}{\E |X(1) |^q}, \quad t \in \T,
\end{align*}
and using $\E|X(1/\lambda)|^q>0$, we get from \eqref{prop:nomfgenonS:proof:eq1} that
\begin{equation*}
\frac{f_q(\lambda t_1)}{f_q(\lambda)} = g_q(t_1).
\end{equation*}
Similarly, since $\lambda t_1, \lambda t_1 t_2 \in \Lambda$ and $1/(\lambda t_1) \in \mS$ we have
\begin{equation*}
\frac{f_q(\lambda t_1 t_2 )}{f_q(\lambda t_1)} = g_q(t_2) \quad \text{ and } \quad \frac{f_q(\lambda t_1 t_2 )}{f_q(\lambda)} = g_q(t_1 t_2).
\end{equation*}
We conclude that
\begin{equation}\label{prop:proof:cauchyfun}
g_q(t_1)  g_q(t_2) = \frac{f_q(\lambda t_1)}{f_q(\lambda)} \frac{f_q(\lambda t_1 t_2 )}{f_q(\lambda t_1)}  =  g_q(t_1 t_2)
\end{equation}
for any $t_1, t_2 \in (0,\infty)$. The joint measurability and Fubini's theorem imply that $t\mapsto \E |X(t) |^q$, and hence $t\mapsto g_q(t)$ is measurable. Hence, for each $q \in (0,\overline{q} \left(X\right))$, there is $\tau(q)\in \R$ such that $g_q(t)=t^{\tau(q)}$ for $t\in (0,\infty)$.

We now show that $q \mapsto \tau(q)$ is a linear function (see \cite{mandelbrot1997mmar}). Let $q_1,q_2 \in (0, \overline{q}(X))$, $w_1, w_2 \geq 0$, $w_1+w_2=1$ and put $q=q_1 w_1 + q_2 w_2$. From H\"older's inequality we have that
\begin{equation*}
  \E|X(t)|^q \leq \left( \E|X(t)|^{q_1} \right)^{w_1} \left( \E|X(t)|^{q_2} \right)^{w_2}
\end{equation*}
and by taking logarithms
\begin{align*}
&\tau(q) \log t + \log \E |X(1) |^q \\
&\hspace{3em} \leq \left( w_1 \tau(q_1) + w_2 \tau(q_2) \right) \log t  +  w_1 \log \E |X(1) |^{q_1} + w_2 \log \E |X(1) |^{q_2}.
\end{align*}
Dividing by $\log t < 0$, $t<1$, and letting $t \to 0$ gives $\tau(q) \geq w_1 \tau(q_1) + w_2 \tau(q_2)$ showing that $\tau$ is concave. But if we divide by $\log t$, $t>1$, and let $t\to \infty$ we get that $\tau(q) \leq w_1 \tau(q_1) + w_2 \tau(q_2)$, hence $\tau$ must be linear.

For $\lambda \in \Lambda$, we have from \eqref{mfdef:prop}
\begin{equation*}
\E|X(\lambda)|^q = \E |M(\lambda, 1)|^q \E|X(1)|^q
\end{equation*}
and so $\E |M(\lambda, 1)|^q = \lambda^{\tau(q)}$ with $\tau$ linear. In particular, the Mellin transform  is $\mathcal{M}_{M(\lambda, 1)}(q) = \lambda^{\tau(q)}$ for every $q \in (0, \overline{q}(X))$. It follows that $M(\lambda, 1)$ is constant a.s. From Proposition \ref{prop:LamSwhendeterministic} we conclude that $X$ is self-similar.
\end{proof}

\begin{remark}\label{rem:cauchyfuneqproof}
The assumption on joint measurability is used only to preclude the existence of pathological solutions of the Cauchy functional equation \eqref{prop:proof:cauchyfun}. Alternatively, one could assume e.g.~that $t\mapsto \E |X(t) |^q$ is continuous at a point or bounded (see e.g.~\cite{aczel1989functional}).
\end{remark}

\begin{remark}
In Proposition \ref{prop:nomfgenonS} we assume $\T$ is not bounded. If $\T$ is bounded, e.g.~$\T=[0,T]$, it is possible to have random scaling over the whole time domain, i.e.~$\mS=\T$. This happens for cascade processes (see Subsection \ref{subsec:examples}).
\end{remark}

Note that the independence of the process and the family of scaling factors in \eqref{mfdef:prop} is crucial for the proof of Proposition \ref{prop:nomfgenonS}. Indeed, one can construct a process with $\T=(0,\infty)$ such that \eqref{mfdef:prop} holds with $\mS=\T$, but the scaling family is not independent of the process (see Section \ref{sec4} and Remark \ref{rem:depconstr}).

\subsection{Summary}

Proposition \ref{prop:nomfgenonS} shows that $\mS$ must not be too large, but we also do not want it to be too small. Based on what we have proved in this section, it is not very restrictive to assume the following.

\begin{assumption}
For any multifractal process, we assume that
	\begin{enumerate}[(i)]
	\item either $\Lambda=(0,1]$ or $\Lambda=[1,\infty)$, unless one of the cases is specified,
	\item $\lambda \mS = \{\lambda t : t \in \mS \} \subset \mS$ for every $\lambda \in \Lambda$.
	\end{enumerate}
\end{assumption}

Moreover, we will implicitly exclude self-similar process from the discussion on multifractal processes.  Note that \textit{(ii)} is satisfied as soon as $\Lambda \subset \mS$. On the other hand, if \textit{(ii)} holds and $1\in \mS$, then $\Lambda \subset \mS$. This assumption leads to two typical classes of multifractal processes:
\begin{itemize}
\item $\Lambda=(0,1]$, $\mS=(0,1]$ or $\mS=(0,T]$, $\mathcal{T}=(0,\infty)$ or $\mathcal{T}=(0,T]$
\item $\Lambda=[1,\infty]$, $\mS=[1,\infty)$ or $\mS=[T,\infty)$, $\mathcal{T}=(0,\infty)$ or $\mathcal{T}=[T,\infty)$.
\end{itemize}
These two classes are closely related, as the following lemma shows by establishing a simple correspondence between them.

\begin{lemma}\label{lemma:Lambdacorrespondance}
	Suppose $X=\{X(t), \, t\in\T\}$ is multifractal with the family of scaling factors $\{M(\lambda,t), \, \lambda \in \Lambda, t\in \mS\}$. Then the process $\overline{X}=\{\overline{X}(t), t \in \overline{\T}\}$ defined by
	\begin{equation*}
	\overline{X}(t)=X(1/t), \quad t \in \overline{\T}=\{1/t : t \in \T\},
	\end{equation*} 
	is multifractal with the family of scaling factors
	\begin{equation*}
	\{\overline{M}(\lambda,t), \, \lambda \in \overline{\Lambda}, t \in \overline{\mS}\}\overset{d}{=}\{M(1/\lambda,1/t), \, \lambda \in \overline{\Lambda}, t \in \overline{\mS}\},
	\end{equation*}
	where $\overline{\mS}=\{1/t : t \in \mS\}$ and  $\overline{\Lambda}=[1,\infty)$ if $\Lambda=(0,1]$ or $\overline{\Lambda}=(0,1]$ if $\Lambda=[1,\infty)$.
\end{lemma}

\begin{proof}
	This is obvious since
	\begin{align*}
	\left\{\overline{X}(\lambda t) \right\}_{t\in \overline{\mS}} &= \left\{ X\left( 1/(\lambda t) \right) \right\}_{1/t \in \mS}\\
	&\overset{d}{=} \left\{ M\left( 1/\lambda, 1/t \right) X \left(1/t \right) \right\}_{1/t \in \mS}\\ &=\left\{ \overline{M}(\lambda,t) \overline{X}(t) \right\}_{t\in \overline{\mS}}.
	\end{align*}
\end{proof}

We also note that it is possible to extend the sets $\mathcal{S}$ and $\mathcal{T}$ by scaling the time with some fixed scale $T>0$. For example, if $\{X(t), \, t \in (0,1]\}$ is multifractal with $\Lambda=(0,1]$ and $\mathcal{S}=(0,1]$, then the process $\{\widetilde{X}, \, t \in (0,T]\}$ defined by
\begin{equation}\label{e:extendingTandS}
\widetilde{X}(t)=X(t/T), \quad t \in (0,T],
\end{equation}
is multifractal with $\Lambda=(0,1]$ and $\mathcal{S}=(0,T]$.

\section{Properties and examples}\label{sec3}

Our first goal is to derive general properties of the family $M=\{M(\lambda, t), \, \lambda \in \Lambda, t \in \mS\}$, dictated by the relation \eqref{mfdef:prop}. As in the previous section, certain regularity assumptions are needed for the proofs.

\begin{proposition}\label{prop:Mproperties}
If $X=\{X(t), \, t \in \T\}$ is multifractal and $\overline{q} \left(X\right)>0$, then the following holds:
\begin{enumerate}[(i)]
  \item $M(1, t)=1$ a.s.~for every $t \in \mS$.
  \item For every $\lambda_1, \lambda_2 \in \Lambda$ and $t\in \mS$, $M(\lambda_1 \lambda_2,t) \overset{d}{=} M^{(1)} M^{(2)}$ with $M^{(1)}\overset{d}{=} M(\lambda_1, \lambda_2 t)$ and $M^{(2)}\overset{d}{=} M(\lambda_2, t)$ independent.
  \item Let $\lambda \in \Lambda$ and $t\in \mS$. Then for every $n \in \N$ there exist independent identically distributed positive random variables $M^{(1)}, \dots, M^{(n)}$ such that
  \begin{equation*}
    M(\lambda, t) \overset{d}{=} M^{(1)} \cdots M^{(n)}.
  \end{equation*}
  Moreover, $M^{(1)}\overset{d}{=} M(\lambda^{1/n}, t)$.
\end{enumerate}
\end{proposition}

\begin{proof}
\textit{(i)} By the same argument as in the proof of Proposition \ref{prop:nomfgen}, for $t\in \mS$ we obtain from $X(t)\overset{d}{=} M(1,t) X(t)$ that for $z \in (0,\overline{q} \left(X\right))$
\begin{equation*}
  \mathcal{M}_{M(1,t)}(z) = 1,
\end{equation*}
which implies the statement.

\textit{(ii)} Since $\lambda_1, \lambda_2 \in \Lambda$, then $\lambda_1 \lambda_2 \in \Lambda$, $\lambda_2 t \in \mS$ and from \eqref{mfdef:prop} we have that for $t\in \mS$
\begin{align*}
  \mathcal{M}_{|X(\lambda_1 \lambda_2 t)|}(z) &= \mathcal{M}_{M(\lambda_1 \lambda_2, t)}(z) \mathcal{M}_{|X(t)|}(z), \\
  \mathcal{M}_{|X(\lambda_1 \lambda_2 t)|}(z) &= \mathcal{M}_{M(\lambda_1, \lambda_2 t)}(z) \mathcal{M}_{|X(\lambda_2 t)|}(z), \\
  \mathcal{M}_{|X(\lambda_2 t)|}(z) &= \mathcal{M}_{M(\lambda_2,t)}(z) \mathcal{M}_{|X(t)|}(z),
\end{align*}
and therefore
\begin{equation*}
  \mathcal{M}_{M(\lambda_1 \lambda_2, t)}(z) \mathcal{M}_{|X(t)|}(z) = \mathcal{M}_{M(\lambda_1, \lambda_2 t)}(z) \mathcal{M}_{M(\lambda_2, t)}(z) \mathcal{M}_{|X(t)|}(z).
\end{equation*}
As in Proposition \ref{prop:nomfgen}, it follows for $z \in (0,\overline{q} \left(X\right))$ that
\begin{equation*}
  \mathcal{M}_{M(\lambda_1 \lambda_2, t)}(z) = \mathcal{M}_{M(\lambda_1, \lambda_2 t)}(z) \mathcal{M}_{M(\lambda_2,t)}(z).
\end{equation*}
Hence, taking $M^{(1)}\overset{d}{=} M(\lambda_1, \lambda_2 t)$ and $M^{(2)}\overset{d}{=} M(\lambda_2, t)$ independent completes the proof.

\textit{(iii)} Similarly as in \textit{(ii)}, $\lambda \in \Lambda$ implies $\lambda^{1/n} \in \Lambda$ and from \eqref{mfdef:prop} it follows for $i=0,\dots,n-1$ that
\begin{equation*}
  \mathcal{M}_{\left| X\left( \lambda^{(n-i)/n} t\right) \right|}(z) = \mathcal{M}_{M\left(\lambda^{1/n},  \lambda^{(n-i-1)/n} t \right)}(z) \mathcal{M}_{\left| X\left( \lambda^{(n-i-1)/n} t\right) \right|}(z).
\end{equation*}
A successive application and \eqref{Mstatmarg} yield
\begin{equation*}
  \mathcal{M}_{M(\lambda, t)}(z) \mathcal{M}_{|X(t)|}(z) = \left( \mathcal{M}_{M(\lambda^{1/n}, t)}(z) \right)^n \mathcal{M}_{|X(t)|}(z)
\end{equation*}
and by the same argument as in \textit{(ii)} taking $M^{(1)}, \dots, M^{(n)}$ independent and distributed as $M(\lambda^{1/n},t)$ gives $M(\lambda,t) \overset{d}{=} M^{(1)} \cdots M^{(n)}$.
\end{proof}

\begin{remark}
The property \textit{(iii)} in Proposition \ref{prop:Mproperties} appears across the literature under various names. In \cite{veneziano1999basic} it is referred to as log-infinite divisibility (see also \cite{bacry2003log}). In different context, the authors of \cite{hirsch2013mellin} obtain a random variable possessing the same property which is called multiplicative infinite divisibility there. Zolotarev \cite[Section 3.5]{zolotarev1986one} refers to the same property as $M$-infinite divisibility. Clearly, Proposition \ref{prop:Mproperties}\textit{(iii)} implies that for every $\lambda \in \Lambda$ and $t\in \mS$, $\log M(\lambda,t)$ is infinitely divisible.
\end{remark}

The infinite divisibility of $\log M(\lambda,t)$ suggest an intimate relation with L\'evy processes. Recall that a L\'evy process is a process starting at zero with stationary independent increments and stochastically continuous. Recall that the stochastic continuity of some process $\{Y(t)\}$ means that for every $t_0$, $Y(t) \to^P Y(t_0)$ as $t \to t_0$. If the same holds with convergence in distribution, then we say that the process is continuous in law.

\begin{proposition}\label{prop:MtoL}
Suppose $X=\{X(t), \, t \in \T\}$ is multifractal, $\overline{q} \left(X\right)>0$ and for some $t\in \mS$ (and hence for every $t\in\mS$), $\{M(\lambda,t), \, \lambda \in \Lambda\}$ is continuous in law. Let $\{L(s), \, s \geq 0\}$ be a L\'evy process such that
\begin{equation*}
  L(1) \overset{d}{=} \begin{cases}
  \log M(e^{-1}, t),& \text{ if } \Lambda=(0,1], \\
  \log M(e, t),& \text{ if } \Lambda=[1,\infty).
  \end{cases}
\end{equation*}
Then for every $s\geq 0$
\begin{equation}\label{MtoL01}
L(s) \overset{d}{=} \log M(e^{-s}, t)
\end{equation}
if $\Lambda=(0,1]$, and if $\Lambda=[1,\infty)$
\begin{equation}\label{MtoL1infty}
L(s) \overset{d}{=} \log M(e^s, t).
\end{equation}
\end{proposition}

\begin{proof}
Let $\mu$ denote the distribution of $L(1)$ which is infinitely divisible by Proposition \ref{prop:Mproperties}\textit{(iii)}. For every $n\in \N$ there is a unique probability measure $\mu^{1/n}$ such that the $n$-fold convolution denoted by $\left(\mu^{1/n}\right)^n$ is $\mu$. Moreover, $\mu^s$ is well-defined for every $s\geq 0$ and is equal to the distribution of $L(s)$ (see \cite[Chapter 7]{sato1999levy}). For $s\geq 0$, let $\mu_s$ denote the distribution of $\log M(e^{-s},t)$ if $\Lambda=(0,1]$ or the distribution of $\log M(e^s,t)$ if $\Lambda=[1,\infty)$. Clearly, $\mu_1=\mu$ and by Proposition \ref{prop:Mproperties}\textit{(iii)} for every $n\in \N$, $\mu_1=\left(\mu_{1/n}\right)^n$ and so $\mu_{1/n}=\mu^{1/n}$. Again by Proposition \ref{prop:Mproperties}\textit{(iii)} $\mu_{m/n}=\left(\mu_{1/n} \right)^m = \mu^{m/n}$ for every $m\in \N$. This proves the statement for every rational $s$.
If $s$ is irrational, take $(s_n)$ to be a sequence of rational numbers such that $s_n\to s$. By the continuity in law $\mu_{s_n} \to^d \mu_s$ and so $\mu_s=\mu^s$.
\end{proof}

\begin{remark}
One can avoid making assumptions on the family $\{M(\lambda,t), \, \lambda\in \Lambda\}$ by using conditions on the original process. Suppose $X$ is continuous in law. For some $t\in \mS$ we have $P(X(t)\neq 0)>0$ and for the sequence $(\lambda_n)$ in $\Lambda$ such that $\lambda_n \to \lambda \in \Lambda$ we have $X(\lambda_n t) \to^d X(\lambda t)$ and so $M(\lambda_n,t) X(t) \to^d M(\lambda, t) X(t)$. From here we obtain for $\theta \in \R$
\begin{equation*}
  \E \left[ \mathbf{1}_{\{|X(t)|>0\}} e^{  i \theta \left(\log M(\lambda_n, t) + \log |X(t)| \right) } \right] \to \E \left[ \mathbf{1}_{\{|X(t)|>0\}} e^{ i \theta \left(\log M(\lambda, t) + \log |X(t)|\right)} \right].
\end{equation*}
By independence it follows that
\begin{align*}
  &\E \left[ e^{ i \theta \log M(\lambda_n, t) } \right]  \E \left[ \mathbf{1}_{\{|X(t)|>0\}} e^{ i \theta \log |X(t)| } \right]\\
  &\hspace{4cm} \to \E \left[ e^{ i \theta \log M(\lambda, t) } \right]  \E \left[ \mathbf{1}_{\{|X(t)|>0\}} e^{ i \theta \log |X(t)| } \right].
\end{align*}
Assuming additionally that the characteristic function of $\log |X(t)|$ has only isolated zeros, we can conclude that
\begin{equation*}
  \E \left[ e^{ i \theta \log M(\lambda_n, t) } \right] \to \E \left[ e^{ i \theta \log M(\lambda, t) } \right],
\end{equation*}
yielding continuity in law of $\{\log M(\lambda,t), \, \lambda \in \Lambda\}$. The assumption on the moments in Propositions \ref{prop:nomfgen}, \ref{prop:nomfgenonS} and \ref{prop:Mproperties} can be replaced with the condition that the characteristic function of $\log |X(t)|$ has only isolated zeros.
\end{remark}

\begin{remark}\label{remark:afterprops}
Notice that in the proofs of Propositions \ref{prop:LamSwhendeterministic}-\ref{prop:MtoL} only one-dimensional distributions of the multifractal process $X$ are used. Therefore it is enough to assume $X$ is marginally multifractal and hence it applies to processes satisfying only \eqref{mfdef:motivating}. 
\end{remark}

We note that \eqref{MtoL01} and \eqref{MtoL1infty} only show that the one-dimensional marginal distributions of two processes are equal. However, $\{\log M(e^{-s}, t), \, s \geq 0\}$  (or $\{\log M(e^{s}, t), \, s \geq 0\}$) need not be a L\'evy process. For Brownian motion such examples are known under the name fake Brownian motion (see \cite{oleszkiewicz2008fake} and references therein). Actually, we will show in Subsection \ref{subsec:examples} that the well-known example of multifractal process, multiplicative cascades, provide a family of scaling factors that does not arise from some L\'evy process.

Proposition \ref{prop:MtoL} shows that the marginal distributions of the family $\{M(\lambda,t)\}$ are completely determined by $M(e^{-1},t)$ (or $M(e,t)$). It also provides an approach for constructing a family of scaling factors with properties as in Proposition \ref{prop:Mproperties}. Indeed, for fixed $t> 0$ one could take $M(\lambda, t)=e^{ L(-\log \lambda)}$ for $\Lambda=(0,1]$ or $M(\lambda,t)=e^{ L(\log \lambda)}$ for $\Lambda=[1,\infty)$ with $\{L(s), \, s \geq 0\}$ being some L\'evy process. This idea will be further developed in Section \ref{sec4} where it is used to define a new class of multifractal processes.

\subsection{Scaling of moments}
The scaling of moments in the sense of relation \eqref{momscal:motivating} is a direct consequence of \eqref{mfdef:prop}. Indeed, suppose $X=\{X(t), \, t \in \T\}$ is multifractal and the assumptions of Proposition \ref{prop:MtoL} hold. Assume that $\Lambda=(0,1]$, the argument is similar in the other case. There exists a L\'evy process $\{L(s)\}$ such that $M(\lambda,t) =^d e^{L(-\log \lambda)}$. Let $\Psi$ denote the characteristic exponent of $L$, that is $\Psi(\theta) = \log \E \left[ e^{i\theta L(1)} \right]$. If we assume that $1\in \mS$, then since for $t\in \Lambda$, $X(t) \overset{d}{=} M(t,1) X(1)$, it follows that for $q \in [0, \overline{q}(X))$
\begin{equation*}
\E \left[ M(t,1)^q \right] = \E \left[ e^{q L(-\log t)} \right] < \infty.
\end{equation*}
Hence the moment generating function of $L(s)$ exists on $[0, \overline{q}(X))$ for every $s\geq 0$. Moreover, for $q \in [0, \overline{q}(X))$
\begin{equation*}
  \E \left[ e^{q L(s)} \right] = e^{s \psi(q)}
\end{equation*}
and by naturally extending $\Psi$ we have $\psi(q) = \Psi(-iq)$. We will refer to $\psi$ as the Laplace exponent. The same argument applies to moments of negative order, that is for $q \in (\underline{q}(X),0]$ where $\underline{q}(X)$ is defined in \eqref{qunderline}. This way we have proved:

\begin{proposition}\label{prop:momscalgeneral}
Under the assumptions of Proposition \ref{prop:MtoL}, if $1\in \mS$, then for every $q \in (\underline{q}(X), \overline{q}(X))$
\begin{itemize}
\item if $\Lambda=(0,1]$
\begin{equation*}
  \E |X(t)|^q = t^{-\psi(q)} \E|X(1)|^q, \quad t \in (0,1],
\end{equation*}
\item if $\Lambda=[1,\infty)$
\begin{equation*}
  \E |X(t)|^q = t^{\psi(q)} \E|X(1)|^q, \quad t \in [1,\infty),
\end{equation*}
\end{itemize}
where $\psi$ is the Laplace exponent of the L\'evy process $L$ defined in Proposition \ref{prop:MtoL}.
\end{proposition}

We conclude that the role of the scaling function $\tau$ in \eqref{momscal:motivating} is taken by the Laplace exponent $\psi$ or $-\psi$. Since $\psi$ is the cumulant generating function of $L(1)$, it is well known that $\psi$ is convex and strictly convex if and only if $L(1)$ is non-degenerate. The strict concavity of the scaling function is a typical property characterizing multifractals that satisfy \eqref{momscal:motivating} (see e.g.~\cite{mandelbrot1997mmar}). This corresponds to our case $\Lambda=(0,1]$ when we have $\tau(q)=-\psi(q)$ which is strictly concave if $L(1)$ is non-degenerate.

\begin{remark}\label{rem:noscalingofmom}
The scaling of moments as in \eqref{momscal:motivating} cannot hold for every $t>0$ with $\tau$ being nonlinear. Indeed, this follows as in the proof of Proposition \ref{prop:nomfgenonS} (see also \cite{mandelbrot1997mmar}).
\end{remark}

Without involving moments, the scaling property may be expressed in terms of the Mellin transforms. Assuming $1\in \mS$, it follows from \eqref{mfdef:prop} that for every $\theta \in \R$ we have
\begin{itemize}
\item if $\Lambda=(0,1]$
\begin{equation*}
  \mathcal{M}_{|X(t)|}(\theta i) = t^{-\Psi(\theta)} \mathcal{M}_{|X(1)|}(\theta i), \quad t \in (0,1],
\end{equation*}
\item if $\Lambda=[1,\infty)$
\begin{equation*}
  \mathcal{M}_{|X(t)|}(\theta i) = t^{\Psi(\theta)} \mathcal{M}_{|X(1)|}(\theta i), \quad t \in [1,\infty).
\end{equation*}
\end{itemize}

\subsection{Examples}\label{subsec:examples}
A prominent example of a truly multifractal process satisfying scale invariance in the sense of \eqref{mfdef:motivatingfdd} are multiplicative cascades. The cascades have been introduced by Mandelbrot \cite{mandelbrot1972} using a discrete grid-based construction. Several equivalent constructions have been proposed to obtain continuous scaling properties starting with \cite{barral2002multifractal} and followed by \cite{bacry2003log}, \cite{muzy2002multifractal} and, more recently, \cite{barral2014exact}.

Let $\nu$ be an arbitrary infinitely divisible distribution and $\Psi$ its characteristic exponent, $\Psi(\theta) = \log \E \left[ e^{i\theta \nu} \right]$. Assume that $\theta_c= \sup \{ \theta \geq 0 : \E \left[ e^{\theta \nu} \right] < \infty \} > 1$ so that the Laplace exponent $\psi (\theta) = \log \E \left[ e^{\theta \nu} \right]$ is finite on $[0, \theta_c)$. Furthermore, assume that $\psi(1)=0$ so that $\E \left[ e^{\nu} \right]= 1$. Next, let $\mathcal{L}$ be an independently scattered infinitely divisible random measure on the half-plane $\mH = \{ (u, v) : u \in \R, v \geq 0\}$ associated to $\nu$ with control measure $\mu(du,dv)=v^2 du dv$ (see \cite{rajput1989spectral} for details). In particular, for every Borel set $A \subset \mH$ such that $\mu(A) < \infty$
\begin{equation}\label{exa:cascade:lambda}
  \E \exp \left\{i \theta \mathcal{L} (A) \right\} = e^{\Psi(\theta) \mu(A)}.
\end{equation}
Fix $T > 0$ and for $t \in \R$ and $l > 0$ define sets (cones)
\begin{equation*}
 A_l(t) = \{(u,v) : v \geq l, \ -f(v)/2 < u-t \leq  f(v)/2 \},
\end{equation*}
where
\begin{equation*}
f(v)= \begin{cases}
v, \ v\leq T,\\
T, \ v> T.
\end{cases}
\end{equation*}
Now we can define stochastic process
\begin{equation*}
 \omega_l(t) = \mathcal{L} \left( A_l(t) \right), \quad t \in \R,
\end{equation*}
and for $l > 0$ a random measure on $\R$ by
\begin{equation*}
  Q_l(dt) = e^{\omega_l(t)} dt.
\end{equation*}
One can show that a.s.~$Q_l$ converges weakly to a random measure $Q$, as $l\to 0$ (see \cite{barral2014exact} for details). This limiting measure $Q$ is called the log-infinitely divisible cascade and the cascade process $\{X(t),\, t \in [0, \infty)\}$ is obtained by putting $X(t) = Q([0, t])$. For $\lambda \in (0, 1]$ and $l \in (0, T]$ the process $\{\omega_l(t)\}$ satisfies the following property
\begin{equation*}
  \{ \omega_{\lambda l}(\lambda t) \}_{t\in [0,T]} \overset{d}{=} \{ \Omega(\lambda) + \omega_l(t)\}_{t\in [0,T]},
\end{equation*}
where $\Omega(\lambda)$ is independent of $\{ \omega_l(t) \}$, it does not depend on $t$ and $l$ and has the characteristic function $\E \left[ e^{i \theta \Omega(\lambda)} \right] = \lambda^{-\Psi(\theta)}$. From here, one easily obtains that for $\lambda \in (0, 1]$
\begin{equation}\label{cascadesscaling}
  \{ X(\lambda t)\}_{t \in [0,T]} \overset{d}{=} \{ W(\lambda) X(t) \}_{t \in [0,T ]},
\end{equation}
with $W(\lambda) = \lambda e^{\Omega(\lambda)}$, independent of $\{ X(t)\}$. This implies that $\{X(t)\}$ is multifractal by Definition \ref{def:mf} with $\T=[0,\infty)$, $\mS=[0,T]$, $\Lambda = (0, 1]$ and the family of scaling factors $M(\lambda, t) = W(\lambda)$ not depending on $t$.

Let $\{Z(s), \, s \geq 0\}$ be a process defined by $Z(s)=\log W(e^{-s})$. Then by Proposition \ref{prop:MtoL}, the one-dimensional distributions of $Z$ correspond to those of some L\'evy process $L$. Here we can actually compute that
\begin{equation*}
  \E \exp \left\{ i\theta Z(s) \right\} = \E \exp \left\{ i\theta \log \left( e^{-s} e^{\Omega(e^{-s})} \right) \right\} = \exp \left\{ s \Psi(\theta) - i \theta s \right\},
\end{equation*}
and hence $L$ can be identified with the process $\{\widetilde{L}(s)-s\}$ where $\{\widetilde{L}(s)\}$ is L\'evy process with characteristic exponent $\Psi$. However, we will now show that $Z$ is not a L\'evy process. It is sufficient to show for arbitrary $0<s_1<s_2$ that
\begin{equation}\label{exa:cascade:ZLdiff}
  \left( Z(s_1)+s_1, Z(s_2)+s_2 \right) \overset{d}{\neq} \left( \widetilde{L}(s_1), \widetilde{L}(s_2) \right).
\end{equation}
Since $\left( Z(s_1)+s_1, Z(s_2)+s_2 \right) \overset{d}{=} \left( \Omega(e^{-s_1}), \Omega(e^{-s_2}) \right)$, we put $\lambda_1=e^{-s_1}$, $\lambda_2=e^{-s_2}$ and for $t, l \in (0,T]$ we consider the characteristic function of $\left( \omega_{\lambda_1 l} (\lambda_1 t), \omega_{\lambda_2 l} (\lambda_2 t)\right)$:
\begin{equation*}
  \E \exp \left\{ i \left(a_1 \omega_{\lambda_1 l}(\lambda_1 t) + a_2 \omega_{\lambda_2 l}(\lambda_2 t) \right) \right\} = \E \exp \left\{ i \left(a_1 \mathcal{L} (A_{\lambda_1 l}(\lambda_1 t)) + a_2 \mathcal{L} (A_{\lambda_2 l}(\lambda_2 t)) \right) \right\}.
\end{equation*}
Now let
\begin{align*}
  B_1 &= A_{\lambda_1 l}(\lambda_1 t) \backslash A_{\lambda_2 l}(\lambda_2 t),\\
  B_2 &= A_{\lambda_1 l}(\lambda_1 t) \cap A_{\lambda_2 l}(\lambda_2 t),\\
  B_3 &= A_{\lambda_2 l}(\lambda_2 t) \backslash A_{\lambda_1 l}(\lambda_1 t),
\end{align*}
and since these sets are disjoint we have by independence and \eqref{exa:cascade:lambda} that
\begin{align*}
  &\E \exp \left\{ i \left(a_1 \omega_{\lambda_1 l}(\lambda_1 t) + a_2 \omega_{\lambda_2 l}(\lambda_2 t) \right) \right\}\\
  &\hspace{3cm}= \E \exp \left\{ i \left(a_1 \mathcal{L} (B_1) + (a_1+a_2) \mathcal{L} (B_2) + \mathcal{L}(B_3) \right) \right\}\\
  &\hspace{3cm}= \exp \left\{ \Psi(a_1) \mu(B_1) + \Psi(a_1+a_2) \mu(B_2) + \Psi(a_2) \mu(B_3) \right\}.
\end{align*}
The cascades are obtained in the limit when $l\to 0$, so we may assume $l \leq t$. A direct computation shows that
\begin{align*}
  \mu(B_1) &= \log \frac{\lambda_1-\lambda_2}{\lambda_1} + \log \frac{t}{l} + 1,\\
  \mu(B_2) &= \log \frac{1}{\lambda_1-\lambda_2} + \log \frac{T}{t},\\
  \mu(B_3) &= \log \frac{\lambda_1-\lambda_2}{\lambda_2} + \log \frac{t}{l} + 1,
\end{align*}
and hence
\begin{align*}
  &\E \exp \left\{ i \left(a_1 \omega_{\lambda_1 l}(\lambda_1 t) + a_2 \omega_{\lambda_2 l}(\lambda_2 t) \right)  \right\}\\
  &= \exp \left\{ \Psi(a_1) \log \lambda_1^{-1} + \psi(a_2) \log \lambda_2^{-1} \log (\lambda_1-\lambda_2) \left(\Psi(a_1) + \Psi(a_2) - \Psi(a_1+a_2)\right) \right\}\\
  &\quad  \times \exp \left\{\Psi(a_1+a_2) \left(1+\log \frac{T}{l} \right) + \left(1+ \log \frac{t}{l} \right) \left(\Psi(a_1) + \Psi(a_2) - \Psi(a_1+a_2)\right) \right\}.
\end{align*}
Now we can write
\begin{equation*}
  \left( \omega_{\lambda_1 l} (\lambda_1 t), \omega_{\lambda_2 l} (\lambda_2 t)\right) \overset{d}{=} \left( \Omega(\lambda_1), \Omega(\lambda_2) \right) + \left(\omega_l'(t), \omega_l''(t) \right),
\end{equation*}
where the random vectors on the right are independent and $\omega_l'(t) =^d \omega_l''(t) =^d \omega_l(t)$. This implies that the characteristic function of $\left( Z(s_1)+s_1, Z(s_2)+s_2 \right)$ is
\begin{align*}
    \E &\exp \left\{ i \left(a_1 (Z(s_1)+s_1) + a_2 (Z(s_2)+s_2) \right) \right\}\\
    &= \exp \left\{ \Psi(a_1) s_1 + \Psi(a_2) s_2 + \log (e^{-s_1}-e^{- s_2}) \left(\Psi(a_1) + \Psi(a_2) - \Psi(a_1+a_2)\right) \right\}.
\end{align*}
On the other hand, $\{\widetilde{L}(s)\}$ is a L\'evy process so that
\begin{align*}
    \E \exp \left\{ i \left(a_1 \widetilde{L}(s_1) + a_2 \widetilde{L}(s_2) \right) \right\} &= \E \exp \left\{ i \left( (a_1+a_2) \widetilde{L}(s_1) + a_2 \widetilde{L}(s_2) - a_2 \widetilde{L}(s_1)\right) \right\}\\
    &= \exp \left\{ \Psi(a_2) (s_2-s_1) + \Psi(a_1+a_2) s_1 \right\},
\end{align*}
which proves \eqref{exa:cascade:ZLdiff}.

\begin{remark}
This fact is of some independent interest as it provides an example of a process whose one-dimensional marginal distributions are the same as those of some L\'evy process but the process itself is not a L\'evy process. This problem has been considered in the martingale setting for Brownian motion and generally for self-similar processes (see \cite{fan2015mimicking}, \cite{oleszkiewicz2008fake} and the references therein). The example obtained here will be elaborated in more details elsewhere.
\end{remark}

Further examples of multifractal processes can be obtained by compounding the cascade process and some self-similar process. For Brownian motion this gives the so-called multifractal random walk (see \cite{bacry2003log}). These models have gained considerable interest in mathematical finance since they can replicate most of the stylized facts of financial time series. The construction of the cascade process necessarily involves the so-called integral scale $T$ which is hard to interpret in finance. Several extensions have been proposed by letting $T\to\infty$, however, these do not satisfy exact scale  invariance as in \eqref{mfdef:motivating} (see e.g.~\cite{duchon2012forecasting}, \cite{muzy2013random}). Note that this is in accordance with Proposition \ref{prop:nomfgenonS}.

As mentioned in the introduction, in \cite{veneziano1999basic} the following variant of \eqref{mfdef:motivatingfdd} is investigated: for every $\lambda \in \Lambda$, $\{X(t)\} =^d \{M_\lambda X(\lambda t) \}$ with $M_\lambda$ independent of $\{X(t)\}$. The canonical example provided there are the so-called processes with independent stationary increment ratios. When $\Lambda=(0,1]$, these in essence correspond to processes defined by $X(t)=e^{L(\log \left(t/T\right))}$ for $t \in [T,\infty)$, where $L$ is a L\'evy process. For $\lambda \in (0,1]$ we have
\begin{align*}
  X(t) &= e^{L(\log \left(t/T\right)) - L(\log \left(t/T\right) + \log \lambda) + L(\log \left(t/T\right)+\log \lambda)}\\
  &= e^{L(\log \left(t/T\right)) - L(\log \left(t/T\right) + \log \lambda)} X(\lambda t),
\end{align*}
where the two random variables on the right are independent by the independence of increments of $L$. This means that $X(t)=^d M_\lambda X(\lambda t)$ with 
$$M_\lambda=e^{L(\log \left(t/T\right)) - L(\log \left(t/T\right) + \log \lambda)}.$$
Although the distribution of $M_\lambda$ does not depend on $t$ by the stationarity of increments of $L$, it is not the same random variable for every $t$ and the definition does not hold in the sense of equality od finite dimensional distributions. It does however satisfy our Definition \ref{def:mf}, but it cannot be extended to the interval $\T=(0,\infty)$. In Section \ref{sec4} we will present an approach that solves this problem.

We mention the last example of a process we are aware that has exact scale invariance property. The multifractional Brownian motion with random exponent is a process $\{Y(t)\}$ obtained by replacing the Hurst parameter of fractional Brownian motion with some stochastic process $\{S(t),\, t\in \R\}$ (see \cite{ayache2005multifractional} for details). Given a stationary process $\{S(t)\}$ it has been showed in \cite[Theorem 4.1]{ayache2005multifractional} using a wavelet decomposition of the process $\{Y(t), \, t\in \R\}$ that for any $\lambda>0$ it holds that
\begin{equation*}
  Y(\lambda t) \overset{d}{=} \lambda^{S(t)} Y(t), \quad \forall t\in \R.
\end{equation*}
It follows from Proposition \ref{prop:nomfgen} and Remark \ref{remark:afterprops} that if $\overline{q}(Y)>0$, then $S(t)$ is constant a.s.~and hence no new examples of multifractal processes can arise in this way.

To our knowledge, the list of examples of processes satisfying exact scale invariance ends here. In the next section we provide a large class of processes that are multifractal by Definition \ref{def:mf}.

\section{$L$-multifractals}\label{sec4}
One of the main drawbacks of the cascade construction is that one obtains a process defined only on a bounded time interval as the construction necessarily involves an integral scale $T$. Our goal here is to develop an alternative approach that would provide processes defined on the unbounded time intervals and satisfying Definition \ref{def:mf}. 

In our definition of multifractality we made two steps away from the scaling property of cascades \eqref{cascadesscaling} that will make this possible. Firstly, we made transition from \eqref{mfdef:motivatingfdd} to \eqref{mfdef:prop} and secondly, we allowed the scaling to hold over set $\mS$ not necessarily equal to $\T$. In fact, we will detail here the construction of multifractal processes such that $\T=(0,\infty)$ and $\mS=\Lambda=(0,1]$. Moreover, for these processes the property \eqref{mfdef:prop} will actually hold over $\mS=(0,\infty)$, but the family of scaling factors and the process will not be independent in this case. As shown in Proposition \ref{prop:nomfgenonS}, it is not possible to preserve independence and random scaling together with $\mS=(0,\infty)$.

Using this approach, an abundance of examples of processes satisfying Definition \ref{def:mf} can be obtained. Indeed, to any L\'evy process (hence any infinitely divisible distribution) and stationary process, there corresponds one multifractal process. For this class of processes the equality \eqref{MtoL01} (or \eqref{MtoL1infty}) in Proposition \ref{prop:MtoL} will not hold only for the one-dimensional marginals, but for the finite dimensional distributions. Because of this correspondence with L\'evy processes, we will call these processes $L$-multifractals. As shown in Subsection \ref{subsec:examples}, the cascades do not belong to this class.

The method for obtaining multifractal processes is inspired by the idea of Lamperti transformation which provides the correspondence between stationary and self-similar processes. If $\{Y(t), \, t \in \R\}$ is a stationary process and $H\in \R$, then the process $\{X(t),\, t>0\}$ defined by
\begin{equation}\label{e:Lamptransclassical1}
    X(t)=t^H Y(\log t), \quad t>0
\end{equation}
is self-similar with Hurst parameter $H$. Conversely, if $\{X(t), \, t > 0\}$ is $H$-ss, then
\begin{equation}\label{e:Lamptransclassical}
    Y(t)=e^{-tH} X(e^t), \quad t\in \R,
\end{equation}
defines a stationary process. Our next goal is to extend this idea to the multifractal case. The results are stated only for the case $\Lambda=(0,1]$ as the other one is analogous by Lemma \ref{lemma:Lambdacorrespondance}.

Specifying $H$ in the Lamperti transformation corresponds to specifying family $\{M(\lambda,t),\allowbreak \, \lambda \in \Lambda, t\in \mS\}$ in the multifractal case. First we show how to obtain such family satisfying properties given in Proposition \ref{prop:Mproperties} from an arbitrary L\'evy process.

\begin{lemma}\label{lemma:Mdef}
Let $L=\{L(s), \, s \geq 0\}$ be a L\'evy process. For $a \geq 0$ define a family of random variables $\{M^{(a)}(\lambda,t), \, \lambda \in (0,1], t \leq e^{a} \}$ given by
\begin{equation}\label{Madef}
  M^{(a)}(\lambda,t) = e^{L(a-\log t - \log \lambda) - L(a-\log t)}.
\end{equation}
Then there exists a family $\{M(\lambda,t), \, \lambda \in (0,1], t > 0\}$ such that for every $a \geq 0$
\begin{equation}\label{MtoMa:fdd}
  \left\{ M(\lambda,t) \right\}_{\lambda \in (0,1], \, t \in (0, e^{a}]} \overset{d}{=} \left\{ M^{(a)}(\lambda,t) \right\}_{\lambda \in (0,1], \, t \in (0, e^{a}]}
\end{equation}
and satisfying the following properties:
\begin{enumerate}[(i)]
  \item For every $\lambda \in (0,1]$, $\left\{ M(\lambda,e^u), \, u \in \R \right\}$ is a stationary process.
  \item For every $t>0$, $M(1,t)=1$ a.s.
  \item For $\lambda_1,\lambda_2 \in (0,1]$, $t>0$ and $a \geq 0$, $M^{(a)}(\lambda_1, \lambda_2 t)$ and $M^{(a)}(\lambda_2, t)$ are independent and
  \begin{equation*}
    M^{(a)}(\lambda_1 \lambda_2, t) = M^{(a)}(\lambda_1,\lambda_2 t) M^{(a)}(\lambda_2, t).
  \end{equation*}
  Moreover, $M(\lambda_1, \lambda_2 t)$ and $M(\lambda_2, t)$ are independent and
  \begin{equation*}
    M(\lambda_1 \lambda_2, t) \overset{d}{=} M(\lambda_1,\lambda_2 t) M(\lambda_2, t).
  \end{equation*}
\end{enumerate}
\end{lemma}

\begin{proof}
If we denote for $\tau\geq 0$, $\widehat{L}^{(\tau)}(s)=L(\tau+s)-L(\tau)$, then clearly $\{\widehat{L}^{(\tau)}(s)\}_{s\geq 0} \overset{d}{=} \{ L(s)\}_{s\geq 0}$. Taking $a \geq b \geq 0$ and so $e^{a} \geq e^{b}$ we get that
\begin{equation}\label{lemma:Mdef:proof1}
\begin{aligned}
  &\left\{ M^{(a)}(\lambda,t) \right\}_{\lambda \in (0,1], \, t \in (0, e^{b}]}\\
  &= \left\{ e^{ L(a-b + b - \log t - \log \lambda) - L(a-b+b-\log t) } \right\}_{\lambda \in (0,1], \, t \in (0, e^{b}]} \\
  &= \left\{ e^{ L(a-b+b- \log t - \log \lambda) - L(a-b) - \left( L(a-b+b-\log t) - L(a-b)\right) } \right\}_{\lambda \in (0,1], \, t \in (0, e^{b}]} \\
  &= \left\{ e^{ \widehat{L}^{(a-b)}(b-\log t - \log \lambda) - \widehat{L}^{(a-b)}(b-\log t) } \right\}_{\lambda \in (0,1], \, t \in (0, e^{b}]} \\
  &\overset{d}{=} \left\{ e^{ L(b - \log t - \log \lambda) - L(b - \log t) } \right\}_{\lambda \in (0,1], \, t \in (0, e^{b}]} \\
  &=  \left\{ M^{(b)}(\lambda,t) \right\}_{\lambda \in (0,1], \, t \in (0, e^{b}]}.
\end{aligned}
\end{equation}
Given $(\lambda_1,t_1),\dots,(\lambda_n,t_n) \in (0,1] \times (0,\infty)$ we can find $a\geq 0$ such that $\max_{i=1,\dots,n}\allowbreak t_i \leq e^{a}$. Let $\mu_{(\lambda_1,t_1),\dots,(\lambda_n,t_n)}$ denote the distribution of $( M^{(a)}(\lambda_1,t_1),\allowbreak\dots, \allowbreak M^{(a)}(\lambda_n,t_n) )$, which does not depend on $a$ as shown above. The measures $\mu_{(\lambda_1,t_1),\dots,(\lambda_n,t_n)}$, $(\lambda_1,t_1),\dots,(\lambda_n,t_n)\allowbreak \in (0,1] \times (0,\infty)$, $n \in \N$ form a consistent family and by the Kolmogorov extension theorem there exists a two-parameter process $\{M(\lambda,t), \, \lambda \in (0,1], t > 0 \}$ with finite dimensional distributions $\mu_{(\lambda_1,t_1),\dots,(\lambda_n,t_n)}$, $(\lambda_1,t_1),\dots,(\lambda_n,t_n) \in (0,1] \times (0,\infty)$, $n \in \N$ such that \eqref{MtoMa:fdd} holds. We now prove properties \textit{(i)}-\textit{(iii)}.

\textit{(i)} It is enough to show that for arbitrary $h>0$ and $a \geq 0$ it holds that
\begin{equation*}
  \left\{ M(s,e^{u}) \right\}_{u \leq a} \overset{d}{=} \left\{ M(s,e^{u-h}) \right\}_{u \leq a}.
\end{equation*}
Since $u \leq a$ implies $e^{u-h}\leq e^u\leq e^{a}$ we have
\begin{align*}
  \left\{ M(s,e^{u-h}) \right\}_{u \leq a} &\overset{d}{=} \left\{ M^{(a)}(\lambda,e^{u-h}) \right\}_{u \leq a}\\
  &= \left\{ e^{L(a-u+h - \log \lambda) - L(a-u+h)} \right\}_{u \leq a}\\
  &= \left\{ e^{\widehat{L}^{(h)}(a-u - \log \lambda) - \widehat{L}^{(h)}(a-u)} \right\}_{u \leq a}\\
  &\overset{d}{=} \left\{ e^{L(a-u - \log \lambda) - L(a-u)} \right\}_{u \leq a}\\
  &= \left\{ M^{(a)}(\lambda,e^{u}) \right\}_{u \leq a}\\
  &\overset{d}{=} \left\{ M(\lambda,e^{u}) \right\}_{u \leq a}.
\end{align*}

\textit{(ii)} This is clear from \eqref{Madef} and \eqref{MtoMa:fdd}.

\textit{(iii)} By taking $a$ such that $t \leq e^{a}$ we have
\begin{align*}
  M(\lambda_1 \lambda_2, t) &\overset{d}{=}  M^{(a)}(\lambda_1 \lambda_2,t) = e^{L(a - \log t - \log \lambda_1 - \log \lambda_2) - L(a-\log t)}\\
  &= e^{L(a-\log t - \log \lambda_1 - \log \lambda_2) - L(a-\log t - \log \lambda_2) + L(a-\log t - \log \lambda_2)- L(a-\log t)}\\
  &= M^{(a)}(\lambda_1, \lambda_2 t)  M^{(a)}(\lambda_2, t),
\end{align*}
which are independent by the independence of increments of $L$. Since
\begin{align*}
  \left( M(\lambda_1, \lambda_2 t), M(\lambda_2, t)\right) \overset{d}{=}  \left( M^{(a)}(\lambda_1, \lambda_2 t), M^{(a)}(\lambda_2, t) \right),
\end{align*}
the statement follows.
\end{proof}

A family $\{M(\lambda,t), \, \lambda \in (0,1], t > 0\}$ will be said to \textbf{correspond to a L\'evy process} $L$ if for every $a\geq 0$ \eqref{MtoMa:fdd} holds with $\{M^{(a)}(\lambda,t), \, \lambda \in (0,1], t \leq e^{a} \}$ given by \eqref{Madef}. We can use the previous construction of such family to build the multifractal process from a stationary process. The multifractal process obtained in this way will have $\{M(\lambda,t), \, \lambda \in (0,1], t \in (0,1])\}$ as the family of scaling factors. This represents a multifractal analog of the Lamperti transformation.

\begin{theorem}\label{thm:MFLamperti}
Let $L=\{L(s), \, s \geq 0\}$ be a L\'evy process and $Y=\{Y(u), \, u \in \R\}$ a stationary process independent of $L$. For $a \geq 0$ define the process $\{X^{(a)}(t), \, t \in (0,e^{a}] \}$ by setting
\begin{equation}\label{Xadef}
  X^{(a)}(t) = e^{L(a-\log t) - L(a)} Y(\log t).
\end{equation}
Then there exists a process $X=\{X(t), \, t > 0 \}$ such that for every $a \geq 0$
\begin{equation}\label{XtoXa:fdd}
  \left\{ X(t) \right\}_{t \in (0,e^{a}]} \overset{d}{=} \left\{ X^{(a)}(t) \right\}_{t \in (0,e^{a}]}.
\end{equation}
and for every $\lambda \in (0,1]$
\begin{equation}\label{XisMF}
  \left\{ X(\lambda t) \right\}_{t\geq 0} \overset{d}{=} \left\{ M(\lambda,t) X(t) \right\}_{t\geq 0},
\end{equation}
where $\{M(\lambda,t), \, \lambda \in (0,1], t > 0 \}$ is the family of scaling factors corresponding to a L\'evy process $L$.

The process $X$ and the family $\{M(\lambda,t), \, \lambda \in (0,1], t \in (0,1] \}$ are independent, hence $X$ is multifractal with $\mS=\Lambda=(0,1]$.
\end{theorem}

\begin{proof}
Denoting again for $\tau\geq 0$, $\widehat{L}^{(\tau)}(s)=L(\tau+s)-L(\tau)$, we have for $a \geq b \geq 0$
\begin{align*}
  \left\{ X^{(a)}(t) \right\}_{t \in (0,e^{b}]} &= \left\{ e^{L(a-b+b-\log t) - L(a-b+b)} Y(\log t) \right\}_{t \in (0, e^{b}]}\\
  &= \left\{ e^{L(a-b+b-\log t)-L(a-b) - \left(L(a-b+b)- L(a-b) \right)} Y(\log t) \right\}_{t \in (0, e^{b}]}\\
  &= \left\{ e^{\widehat{L}^{(a-b)}(b-\log t) - \widehat{L}^{(a-b)}(b)} Y(\log t) \right\}_{t \in (0, e^{b}]}\\
  &\overset{d}{=} \left\{ e^{L(b-\log t) - L(b)} Y(\log t) \right\}_{t \in (0, e^{b}]}\\
  &= \left\{ X^{(b)}(t) \right\}_{t \in (0, e^{b}]}.
\end{align*}
As in Lemma \ref{lemma:Mdef}, this shows that finite dimensional distributions of $\{X^{(a)}(t), \, t \leq e^{a} \}$, $a\geq 0$ do not depend on $a$ and form a consistent family. An appeal to the Kolmogorov extension theorem gives the existence of $X$ satisfying \eqref{XtoXa:fdd}.

To show multifractality of $X$ it suffices to show \eqref{XisMF} holds over $t \in (0,e^{a}]$ for arbitrary $a\geq 0$. By using the notation of Lemma \ref{lemma:Mdef} and stationarity of $Y$ we have
\begin{align}
  \left\{ X(\lambda t) \right\}_{t \in (0,e^{a}]} &\overset{d}{=} \left\{ X^{(a)}(\lambda t) \right\}_{t \in (0,e^{a}]}\nonumber \\
  &= \left\{ e^{L(a-\log t - \log \lambda)-L(a)} Y(\log t + \log \lambda) \right\}_{t \in (0,e^{a}]}\nonumber \\
  &= \left\{ e^{L(a-\log t - \log \lambda)-L(a-\log t)} e^{L(a-\log t)-L(a)} Y(\log t + \log \lambda) \right\}_{t \in (0,e^{a}]}\nonumber \\
  &\overset{d}{=} \left\{ e^{L(a-\log t - \log \lambda)-L(a-\log t)} e^{L(a-\log t)-L(a)} Y(\log t) \right\}_{t \in (0,e^{a}]}\nonumber \\
  &= \left\{ M^{(a)}(\lambda, t) X^{(a)}(t) \right\}_{t \in (0,e^{a}]}  \label{thm:MFLamperti:proofline}\\
  &\overset{d}{=} \left\{ M(\lambda,t) X(t) \right\}_{t \in (0,e^{a}]}.\nonumber
\end{align}
When $t\in(0,1]$, two factors in \eqref{thm:MFLamperti:proofline} are independent due to independence of increments of $L$ and independence of $L$ and $Y$ and hence family $\{M(\lambda,t), \, \lambda \in (0,1], t \in (0,1] \}$ can be taken independent of $\{X(t)\}$.
\end{proof}

A multifractal process $X$ will be said to be \textbf{$L$-multifractal} if its family of scaling factors corresponds to a L\'evy process $L$. Every process $X$ obtained as in Theorem \ref{thm:MFLamperti} is $L$-multifractal.

\begin{remark}\label{rem:depconstr}
Note that the equality of finite dimensional distributions \eqref{XisMF} holds over $(0,\infty)$, but the scaling family is independent from the process only over $(0,1]$, hence $\mS=(0,1]$. Putting $X^{(a)}(t) = e^{L(a+\log t) - L(a)} Y(\log t)$ for $t\in(e^{-a},\infty)$ instead of \eqref{Xadef}, yields by similar arguments a process for which \eqref{XisMF} holds over $(0,\infty)$. The random field $M$ corresponds to the one constructed similarly as in Lemma \ref{lemma:Mdef} but with $M^{(a)}(\lambda,t) = e^{L(a+\log t - \log \lambda) - L(a+\log t)}$, $\lambda \in (0,1]$, $t\in(e^{-a},\infty)$, replacing \eqref{Madef}. The independence does not hold in \eqref{XisMF} hence the process $X$ is not multifractal by Definition \ref{def:mf}.
\end{remark}

The analog of the inverse Lamperti transformation also holds. Indeed, every $L$-multifractal corresponds to a stationary process in a sense given by the following theorem.

\begin{theorem}\label{thm:MFLampertiInv}
Suppose $\{X(t), \, t>0\}$ is $L$-multifractal with the scaling family $\{M(\lambda,t), \,\allowbreak \lambda \in (0,1], t > 0 \}$. Then the process $\{Y(s), \, s \geq 0 \}$ defined by
\begin{equation*}
  Y(s) =  M(e^{-s},e^s) X(e^s).
\end{equation*}
is stationary.
\end{theorem}

\begin{proof}
For arbitrary $h>0$ we have
\begin{align*}
  \left\{ Y(s+h) \right\}_{s \geq 0} &= \left\{ M(e^{-s-h},e^{s+h}) X(e^{s+h}) \right\}_{s \geq 0} \\
  &\overset{d}{=} \left\{ M(e^{-s},e^{s})  M(e^{-h},e^{s+h}) X(e^{s+h}) \right\}_{s \geq 0}\\
  &\overset{d}{=} \left\{ M(e^{-s},e^{s})  X(e^{s}) \right\}_{s \geq 0}\\
  &\overset{d}{=} \left\{ Y(s) \right\}_{s \geq 0}.
\end{align*}
\end{proof}

Let us mention that if in Theorems \ref{thm:MFLamperti} and \ref{thm:MFLampertiInv} $L(s)=-Hs$, $s\geq 0$, then one obtains the classical form of the Lamperti transformation.

\section{Properties of $L$-multifractal processes}\label{sec5}

In this section we derive several properties of $L$-multifractal processes defined in Theorem \ref{thm:MFLamperti}.

\subsection{Scaling of moments}
Since the process $X$ from Theorem \ref{thm:MFLamperti} is multifractal, Proposition \ref{prop:momscalgeneral} implies the scaling of moments holds. By using \eqref{Xadef}, we can actually prove more.

\begin{proposition}
Let $X$ be a process obtained in Theorem \ref{thm:MFLamperti} from L\'evy process $L$ and stationary process $Y$. If $q\in \R$ is such that
\begin{equation*}
\E \left[ e^{qL(1)} \right]< \infty,  \ \E \left[ e^{-qL(1)} \right]< \infty \ \text{ and } \ \E|Y(1)|^q<\infty, 
\end{equation*}
then $\E|X(t)|^q<\infty$ for every $t>0$. If $\psi$ is the Laplace exponent of $L$,  $\E \left[ e^{q L(s)} \right] = e^{s \psi(q)}$, then
\begin{equation}\label{L-MF:moscal}
  \E|X(t)|^q = \begin{cases}
  t^{-\psi(q)} \E|X(1)|^q, & \text{ if } t \leq 1, \\
  t^{\psi(-q)} \E|X(1)|^q, & \text{ if } t > 1.
  \end{cases}
\end{equation}
\end{proposition}

\begin{proof}
By taking $a\geq 0$ such that $t \leq e^a$ we have
\begin{align*}
  \E|X(t)|^q = \E \left[ e^{q\left( L(a-\log t) - L(a) \right)} \right] \E|Y(\log t)|^q &= \begin{cases}
  \E \left[ e^{q L(-\log t)} \right] \E|X(1)|^q, & \text{ if } t \leq 1, \\[1ex]
  \E \left[ e^{-q L(\log t)} \right] \E|X(1)|^q, & \text{ if } t > 1.
  \end{cases}\\
  &= \begin{cases}
  t^{-\psi(q)} \E|X(1)|^q, & \text{ if } t \leq 1, \\
  t^{\psi(-q)} \E|X(1)|^q, & \text{ if } t > 1.
  \end{cases}  
\end{align*}
\end{proof}

It is important to note that \eqref{L-MF:moscal} does not contradict Remark \ref{rem:noscalingofmom}, since in \eqref{L-MF:moscal} we actually have $\E|X(t)|^q = t^{\tau(q, t)} \E|X(1)|^q$ with the exponent $\tau$ depending additionally on $t$:
\begin{equation*}
  \tau(q,t) = \begin{cases}
  -\psi(q), & \text{ if } t \leq 1, \\
  \psi(-q), & \text{ if } t > 1.
  \end{cases}
\end{equation*}

In terms of the Mellin transforms, we similarly obtain for $\theta \in \R$
\begin{equation*}
\mathcal{M}_{|X(t)|}(\theta i) = \begin{cases}
t^{-\Psi(\theta)} \mathcal{M}_{|X(1)|}(\theta i), & \text{ if } t \leq 1, \\
t^{\Psi(-\theta)} \mathcal{M}_{|X(1)|}(\theta i), & \text{ if } t > 1,
\end{cases}
\end{equation*}
where $\Psi$ is the characteristic exponent of $L$, $\Psi(\theta) = \log \E \left[ e^{i\theta L(1)} \right]$.

\subsection{Stationarity of increments}
When it comes to applications like finance, turbulence and other fields, an important feature of stochastic process used for modeling is stationarity of increments. This provides applicability of statistical methods and is often plausible to assume. However, even for self-similar processes this may be hard to achieve. In fact, as noted by \cite{barndorff1999stationary}, there is no simple characterization of marginal laws of self-similar processes with stationary increments. 

We will first show that, unfortunately, the process $X$ constructed in Theorem \ref{thm:MFLamperti} with finite variance can not have stationary increments if considered on the time set $\mathcal{T}=(0,\infty)$. However, if we restrict the time set to, say $\mathcal{T}=(0,1]$, then it may be possible for $X$ to have stationary increments. We will show that by appropriately choosing the stationary process $Y$ in Theorem \ref{thm:MFLamperti} and restricting the time set, one can obtain a multifractal process with second-order stationary increments meaning that its covariance function is the same as if it has stationary increments. 

Suppose $X$ is an $L$-multifractal process defined in Theorem \ref{thm:MFLamperti} with finite variance. By taking $a\geq 0$ such that $e^a\geq t > s$, we have directly from \eqref{Xadef} that
\begin{equation}\label{e:covLMFgen}
\begin{aligned}
&\E X(t)X(s) = \E \left[ e^{L(a-\log t) - L(a)} Y(\log t) e^{L(a-\log s) - L(a)} Y(\log s)\right]\\
&= \E \left[ e^{L(a-\log t) - L(a) + L(a-\log s) - L(a)}\right] \E \left[ Y(\log t) Y(\log s)\right]\\
&=\begin{cases}
\E \left[ e^{-2 \left( L(a) - L(a-\log s)\right)} \right] \E \left[e^{-\left(L(a-\log s) - L(a-\log t)\right) }\right] \E \left[ Y(\log t) Y(\log s)\right], & \text{ if } t > 1, \ s>1, \\
\E \left[ e^{- \left( L(a) - L(a-\log t)\right)} \right] \E \left[ e^{ L(a-\log s) - L(a) }\right] \E \left[ Y(\log t) Y(\log s)\right], & \text{ if } t > 1, \ s\leq 1, \\
\E \left[ e^{2\left( L(a-\log t) - L(a)\right)} \right] \E \left[ e^{L(a-\log s) - L(a-\log t) }\right] \E \left[ Y(\log t) Y(\log s)\right], & \text{ if } t \leq 1, \ s\leq 1, \\
\end{cases}\\
&=\begin{cases}
t^{\psi(-1)} s^{\psi(-2)-\psi(-1)} \E \left[ Y(\log t) Y(\log s)\right], & \text{ if } t > 1, \ s>1, \\
t^{\psi(-1)} s^{-\psi(1)} \E \left[ Y(\log t) Y(\log s)\right], & \text{ if } t > 1, \ s\leq 1, \\
t^{\psi(1)-\psi(2)} s^{-\psi(1)} \E \left[ Y(\log t) Y(\log s)\right], & \text{ if } t \leq 1, \ s\leq 1, \\
\end{cases}
\end{aligned}
\end{equation}
where $\psi$ is the Laplace exponent of $L$. 

On the other hand, if $X$ is $L$-multifractal with stationary increments, then for $t>s$
\begin{equation}\label{e:covLMFsi}
\E X(t)X(s) = \begin{cases}
\frac{1}{2} \left( t^{\psi(-2)} + s^{\psi(-2)} - (t-s)^{\psi(-2)}\right) \E X(1)^2, & \text{ if } t > 1, \ s>1 \ \text{and} \ t-s>1, \\
\frac{1}{2} \left( t^{\psi(-2)} + s^{\psi(-2)} - (t-s)^{-\psi(2)}\right) \E X(1)^2, & \text{ if } t > 1, \ s>1 \ \text{and} \ t-s\leq 1, \\
\frac{1}{2} \left( t^{\psi(-2)} + s^{-\psi(2)} - (t-s)^{\psi(-2)}\right) \E X(1)^2, & \text{ if } t > 1, \ s\leq 1 \ \text{and} \ t-s>1, \\
\frac{1}{2} \left( t^{\psi(-2)} + s^{-\psi(2)} - (t-s)^{-\psi(2)}\right) \E X(1)^2, & \text{ if } t > 1, \ s\leq 1 \ \text{and} \ t-s\leq 1, \\
\frac{1}{2} \left( t^{-\psi(2)} + s^{-\psi(2)} - (t-s)^{-\psi(2)}\right) \E X(1)^2, & \text{ if } t \leq 1,
\end{cases}
\end{equation}
which follows from \eqref{L-MF:moscal} and the following identity valid for any stationary increments process with finite variance
\begin{align*}
\E X(t)X(s) &= \frac{1}{2} \left( \E X(t)^2 + \E X(s)^2 - \E \left(X(t)-X(s)\right)^2 \right)\\
&= \frac{1}{2} \left( \E X(t)^2 + \E X(s)^2 - \E X(t-s)^2 \right).
\end{align*}

We are now considering is it possible to choose $Y$ and $\psi$ in \eqref{e:covLMFgen} to get the covariance function \eqref{e:covLMFsi} as if $X$ has stationary increments. Let $u, h >0$ so that $e^{u+h}>1$, $e^u >1$ and suppose that $e^{u+h}-e^u \leq 1$. Then by equating \eqref{e:covLMFgen} and \eqref{e:covLMFsi} we have
\begin{align}
&\E Y(u+h) Y(u) = \E Y(\log e^{u+h}) Y(\log e^u)\nonumber\\
&\quad=\frac{1}{2} \E X(1)^2 \frac{e^{\psi(-2)(u+h)} + e^{\psi(-2)u} - (e^{u+h}-e^u)^{-\psi(2)}} {e^{\psi(-1)(u+h)} e^{(\psi(-2)-\psi(-1))u}}\nonumber\\
&\quad=\frac{1}{2} \E X(1)^2 \left( e^{(\psi(-2) - \psi(-1))h} + e^{-\psi(-1)h} - e^{-\psi(-1)h} e^{-(\psi(-2) + \psi(2))u} (e^{h}-1)^{-\psi(2)} \right).\label{e:covYudep}
\end{align}
Since $Y$ is stationary, \eqref{e:covYudep} must not depend on $u$, hence it should hold $\psi(-2)=-\psi(2)$. But then $\psi$ is a convex function passing through three collinear points $(-2,\-\psi(2))$, $(0,0)$ and $(2,\psi(2))$ and hence it must be linear (see e.g.~\cite[Lemma 2]{GLST2019Bernoulli}) implying that $L(1)$ is degenerate and $X$ is self-similar. 

To conclude, a process defined in Theorem \ref{thm:MFLamperti} with finite variance cannot have stationary increments unless it is self-similar. One can notice the problem appears with \eqref{e:covLMFsi} having different forms for $t-s>1$ and $t-s\leq 1$. However, if we restrict the time domain of the process to $\mathcal{T}=(0,1]$, then we can obtain a multifractal process with second-order stationary increments. 

Consider a multifractal process from Theorem \ref{thm:MFLamperti} restricted to $\mathcal{T}=(0,1]$. In this case, $X$ can be defined as
\begin{equation*}
X(t)= e^{L(-\log t)} Y(\log t), \quad t \in (0,1],
\end{equation*}
where $L=\{L(t), \, t \geq 0\}$ is some L\'evy process and $Y=\{Y(t),\, t \in \R\}$ is a stationary process. If $u<u+h<0$, then $e^{u+h}\leq 1$, $e^u \leq 1$, $e^{u+h}-e^u \leq 1$ and equating again \eqref{e:covLMFgen} and \eqref{e:covLMFsi} yields
\begin{align}
\E Y(u+h) Y(u) &= \E Y(\log e^{u+h}) Y(\log e^u)\nonumber\\
&=\frac{1}{2} \E X(1)^2 \frac{e^{-\psi(2)(u+h)} + e^{-\psi(2)u} - (e^{u+h}-e^u)^{-\psi(2)}} {e^{(\psi(1)-\psi(2))(u+h)} e^{-\psi(1)u}}\nonumber\\
&=\frac{1}{2} \E X(1)^2 \left( e^{-\psi(1)h} + e^{-(\psi(1)-\psi(2))h} - e^{-(\psi(1) - \psi(2))h} (e^{h}-1)^{-\psi(2)} \right)\nonumber\\
&=\frac{1}{2} \E X(1)^2 e^{-\psi(1)h} \left( 1 + e^{\psi(2)h} - (1 - e^{-h})^{-\psi(2)} \right),\label{e:covSolution}
\end{align}
which does not depend on $u$. Note that for $\psi(q)=-Hq$, $0<H<1$, we recover the covariance function of the stationary process obtained by the classical Lamperti transformation of fractional Brownian motion (see \cite{cheridito2003fractional}). In particular, for $\psi(q)=-q/2$ we get the Ornstein-Uhlenbeck (OU) process (see e.g.~\cite{samorodnitsky1994stable}). Recall that OU process $\{Y(u),\, u \in\R\}$ with parameter $\lambda>0$ is a stationary Gaussian process with mean zero and covariance function
\begin{equation*}
\E \left[ Y(u+h) Y(u) \right]= \E Y(0)^2 e^{-\lambda |h|}, \quad u, h \in \R.
\end{equation*}
Note that it is not immediately clear whether \eqref{e:covSolution} defines a covariance function of some stationary process. We will consider a simple example in the next subsection. Also note that although assuming $\mathcal{T}=(0,1]$ may seem overly restrictive, by using \eqref{e:extendingTandS} we can extend the time set to $(0,T]$ for arbitrary $T>0$.

We summarize the previous discussion in the following proposition.

\begin{proposition}\label{prop:newconstruction}
Let $T>0$ and suppose $L$ is a L\'evy process with Laplace exponent $\psi$ well-defined on $[0,2]$ and 
\begin{equation}\label{e:cov:Y:prop}
\gamma(h)=\frac{1}{2} \E Y(0)^2 e^{-\psi(1)h} \left( 1 + e^{\psi(2)h} - (1 - e^{-h})^{-\psi(2)} \right),
\end{equation}
is a covariance function of strictly stationary process $\{Y(t), \, t \in \R\}$. Then the process $\{X(t), \, t \in (0,T]\}$
\begin{equation*}
X(t) = e^{L\left(-\log (t/T \right)} Y \left(\log (t/T) \right), \quad t \in (0,T],
\end{equation*}
is multifractal with $\Lambda=(0,1]$, $\mathcal{S}=(0,T]$ and
\begin{equation}\label{e:process:cov}
\E X(t)X(s) = \frac{1}{2} T^{\psi(2)} \E X(1)^2 \left( t^{-\psi(2)} + s^{-\psi(2)} - |t-s|^{-\psi(2)}\right), \quad t,s \in (0,T].
\end{equation}
In particular, for any $\varepsilon>0$ the sequence $X(t_j+\varepsilon)-X(t_j)$, $j=1,\dots, \lfloor T/\varepsilon\rfloor$ is weakly stationary with
\begin{equation}\label{e:increment:cov}
\begin{aligned}
\E&\left(X(t_j+\varepsilon)-X(t_j)\right) \left(X(t_i+\varepsilon)-X(t_i) \right)\\
&= \frac{1}{2} \varepsilon^{-\psi(2)} \E X(1)^2 \left( |j-i+1|^{-\psi(2)} + |j-i-1|^{-\psi(2)} - 2 |j-i|^{-\psi(2)} \right).
\end{aligned}
\end{equation}
\end{proposition}

Identity \eqref{e:increment:cov} is easily obtained from \eqref{e:process:cov} and it takes the form of covariances of fractional Gaussian noise (see \cite{samorodnitsky1994stable}). If $\psi(2)=-1$, then the increments are uncorrelated. Furthermore, \eqref{e:process:cov} implies that for any $t,s \in (0,T]$
\begin{equation*}
\E \left(X(t) - X(s)\right)^2 = \frac{1}{2} T^{\psi(2)} \E X(1)^2 |t-s|^{-\psi(2)}.
\end{equation*}
If $-\psi(2)-1>0$, then by Kolmogorov's theorem (see e.g.~\cite[Theorem 2.2.8]{karatzas2012brownian}) there exists a modification of $\{X(t)\}$ which is locally H\"older continuous with exponent $\gamma$ for every $\gamma \in (0, (-\psi(2)-1)/2)$.

\subsection{Examples}

A number of processes can be constructed from Theorem \ref{thm:MFLamperti}. Given a L\'evy process $L$, one can simply take $Y(t)=1$ a.s.~to obtain positive multifractal process which is an exponential of a L\'evy process in logarithmic time extended to the whole $(0,\infty)$. 

We shall consider in more details a specific example that may be viewed as multifractal analog of Brownian motion. Suppose the L\'evy process $L$ is Brownian motion with drift $\mu$ so that $\psi(q)=\mu q + \sigma^2 q^2/2$, $q\in \R$. We consider the process constructed in Proposition \ref{prop:newconstruction} and take stationary process $Y$ to be OU process with parameter $\lambda = \psi(1)+1$, hence
\begin{equation*}
\E \left[ Y(u+h) Y(u) \right]= \E Y(0)^2 e^{-(\psi(1)+1) |h|}, \quad u, h \in \R.
\end{equation*}
Note that this is exactly \eqref{e:cov:Y:prop} with $\psi(2)=-1$. Hence, the process $\{X(t), \, t \in (0,T]\}$ defined in Proposition \ref{prop:newconstruction} will have second-order stationary increments. The condition $\psi(2)=-1$ implies $\mu=-1/2-\sigma^2$ and
\begin{equation*}
\psi(q)=-\left( \frac{1}{2} + \sigma^2 \right) q + \frac{\sigma^2}{2} q^2.
\end{equation*}
The increments of $X$ are uncorrelated and
\begin{equation*}
\E X(t) X(s) = \frac{\min \{t,s\}}{T}.
\end{equation*}

Since the classical Lamperti transformation \eqref{e:Lamptransclassical1} of OU process yields Brownian motion, the process $X$ represents a multifractal analog of Brownian motion. The scaling function function if given by $\tau(q)=-\psi(q)$ and well defined for $q \in (-1,\infty)$ since absolute moments of order less than or equal to one are infinite for Gaussian distribution. The scaling function is of the same form as the scaling function of multifractal random walk which is Brownian motion with time taken to be log-normal multiplicative cascade process (see \cite{bacry2003log}).

Other properties of these processes require deeper investigation which will be addressed in future work. One of the interesting question is whether the sample paths of these processes posses multifractal properties in the sense of the varying local regularity exponents (see e.g.~\cite{abry2015irregularities}, \cite{grahovac2014bounds}, \cite{jaffard1999multifractal} and the references therein).

\appendix

\section{Mellin transform}\label{appendix:mellin}
Recall that the Mellin transform (or Mellin-Stieltjes transform) of a nonnegative random variable $X$ with distribution function $F$ is defined as
\begin{equation*}
  \mathcal{M}_X(z) = \int_0^\infty x^z dF(x)
\end{equation*}
for $z \in \C$. The integral exists for all $z$ in some strip $S=\{z : \sigma_1 \leq \Re z \leq \sigma_2 \}$, $\sigma_1 \leq \sigma_2$ which contains the imaginary axis and possibly degenerates into this axis. The Mellin transform completely determines the distribution of nonnegative random variable $X$. Furthermore, if the strip $S$ does not degenerate into imaginary axis, it is uniquely determined by its values on the interval $(\sigma_1, \sigma_2)$. Indeed, in the case $P(X>0)=1$, by applying the change of variables it is easy to see that $\mathcal{M}_X$ can be expressed as the two-sided Laplace transform of the random variable $-\log X$ with distribution function $G$:
\begin{equation*}
  \mathcal{M}_X(z) = \int_{-\infty}^\infty e^{-zx} dG(x).
\end{equation*}
Since the two-sided Laplace transform is analytic function on $S$, so is $\mathcal{M}_X$ (see e.g.~\cite[p.~240]{widder1946laplace}). Therefore by the familiar property of analytic functions, $\mathcal{M}_X$ is uniquely determined by its values on the interval $(\sigma_1, \sigma_2)$. In the case $P(X=0)>0$, we can apply the same argument to random variable $Y$ defined by distribution function $\widetilde{F}(x)=(F(x)-F(0))/(1-F(0)) \mathbf{1}_{\{x\geq 0\}}$ and use the fact that $\mathcal{M}_X(z)=(1-F(0)) \mathcal{M}_Y(z)$. Moreover, if the strip $S$ does not degenerate into imaginary axis inversion formulas can be obtained by exploiting the relation with the two-sided Laplace transform. The definition can be extended to include real-valued variables, however we do not pursue this question here. More details about Mellin transform can be found in \cite{galambos2004products} and \cite{zolotarev1957mellin}.

The main reason Mellin transform proves useful is the following property: if $X$ and $Y$ are two independent nonnegative random variables and $\mathcal{M}_X$, $\mathcal{M}_Y$ are their Mellin transforms defined on strips $S_1$ and $S_2$ respectively, then the Mellin transform of the product $XY$ in the strip $S_1\cap S_2$ is
\begin{equation*}
  \mathcal{M}_{XY} (z) = \mathcal{M}_X (z) \mathcal{M}_Y (z).
\end{equation*}

\bigskip
\bigskip

\bibliographystyle{agsm}
\bibliography{References}

\end{document}